\documentclass[10pt]{amsart}
\usepackage{amsmath}
\usepackage{amsfonts}
\usepackage{amssymb}
\usepackage{graphicx}
\usepackage{amsthm,graphicx,color,yfonts}
\usepackage{pdfsync}
\usepackage{epstopdf}
\usepackage{marginnote}
\usepackage{esint}

\usepackage[colorlinks=true]{hyperref}
\hypersetup{linkcolor=red,citecolor=blue,filecolor=dullmagenta,urlcolor=darkblue} 

\numberwithin{equation}{section}

\theoremstyle{plain}
\newtheorem{theorem}{Theorem}[section]

\newtheorem{lemma}[theorem]{Lemma}
\newtheorem*{de-lemma}{Lemma}

\theoremstyle{remark}

\theoremstyle{definition}

\DeclareMathOperator{\supp}{supp}

\newcommand{\dd}{\mathrm{d}}
\newcommand{\R}{\mathbb{R}}
\newcommand{\SF}{\mathbb{S}}
\newcommand{\n}{\textbf{\em n}}

\providecommand{\MR}{\relax\ifhmode\unskip\space\fi MR }

\providecommand{\href}[2]{#2}

\begin{document}

\title[Vortex solutions in the Ginzburg-Landau-Painlev\'e theory of phase transition]{Vortex solutions in the Ginzburg-Landau-Painlev\'e theory of phase transition}

\author{Panayotis Smyrnelis} \address[P.~ Smyrnelis]{Institute of Mathematics,
Polish Academy of Sciences, ul. \'{S}niadeckich 8, 00-656 Warsaw, Poland}
\email[P. ~Smyrnelis]{psmyrnelis@impan.pl}


\begin{abstract}
The extended Painlev\'e P.D.E. system $\Delta y -x_1 y - 2 |y|^2y=0$, $(x_1,\ldots,x_n)\in \mathbb{R}^n$, $y:\mathbb{R}^n\to\mathbb{R}^m$, is obtained by multiplying by $-x_1$ the linear term of the Ginzburg-Landau equation $\Delta \eta=|\eta|^2\eta-\eta$,  $\eta:\mathbb{R}^{n}\to\mathbb{R}^{m}$. The two dimensional model $n=m=2$ describes in the theory of light-matter interaction in liquid crystals, the orientation of the molecules at the boundary of the illuminated region. On the other hand, the one dimensional model reduces to the second Painlev\'e O.D.E. $y''-xy-2y^3=0$, $x\in \mathbb{R},$ which has been extensively studied, due to its importance for applications. The solutions of the extended Painlev\'e P.D.E. share some characteristics both with the Ginzburg-Landau equation and the second Painlev\'e O.D.E. The scope of this paper is to construct standard vortex solutions $y:\mathbb{R}^{n}\to\mathbb{R}^{n-1}$ ($\forall n\geq 3$) of the extended Painlev\'e equation. These solutions have in every hyperplane $x_1=\mathrm{Const.}$, a profile similar to the standard vortices $\eta:\mathbb{R}^{n-1}\to\mathbb{R}^{n-1}$ of the Ginzburg-Landau equation, but their amplitude is determined by the Hastings-McLeod solution $h$ of the second Painlev\'e O.D.E. evaluated at $x_1$.

\end{abstract}

\maketitle

\section{The extended Painlev\'e P.D.E.}\label{sec:sec2}
\subsection{Origin of the model}

In the physical context of light-matter interaction in liquid crystals (cf. \cite{clerc2}), we recently discovered that the extended Painlev\'e equation:
\begin{equation}\label{pain 1}
\Delta y-x_1 y-2|y|^2y=0 , \qquad \forall x=(x_1,\ldots,x_n)\in \R^n, y=(y_1,\ldots,y_m):\R^n\to\R^m,
\end{equation}
is relevant to describe the orientation of the molecules at the boundary of the illuminated region (cf. \cite{Clerc2017} when $n=m=1$, \cite{Clerc2018} when $n=m=2$, and \cite{panayotis_4} when $n=2$, $m=1$).
It can alternatively be written as
\begin{equation}
\label{pain 1db}
\Delta y(x)= H_y(x_1,y(x)), \qquad x\in \R^n,
\end{equation}
with a non autonomous potential $H(x_1,y):=\frac{1}{2} x_1 |y|^2 +\frac{1}{2} |y|^4$, $(x_1,y)\in\R\times\R^m$, and $H_y=(\frac{\partial H}{\partial y_1},\ldots,\frac{\partial H}{\partial y_m})\in\R^m$. Also note that equation \eqref{pain 1} has variational structure.
Let
\begin{equation}\label{funcp}
E_{\mathrm{P_{II}}}(u, \Omega)=\int_\Omega \left[ \frac{1}{2} |\nabla u|^2 +\frac{1}{2}  x_1 |u|^2 +\frac{1}{2}| u|^4\right],\ u \in H^1(\Omega;\R^m), \ \Omega\subset \R^n,
\end{equation}
be its associated functional.
In the physical models studied in the aforementioned works, the extended Painlev\'e equation is obtained as the singular limit (after appropriate rescaling) of the system
\begin{equation}\label{sing}
\epsilon^2\Delta u_\epsilon-\mu(x)u_\epsilon-|u_\epsilon|^2u_\epsilon+\epsilon a(\epsilon) f(x)=0 , \ \forall x=(x_1,\ldots,x_n)\in \R^n, u_\epsilon:\R^n\to\R^m, \ \epsilon>0,
\end{equation}
where $\mu(x)=e^{\,-|x|^2}-\chi$, $\chi\in (0,1)$, $f:\R^n\to\R^m$ is a specific map related to $\mu$, and $a(\epsilon)$ is a nonnegative parameter.
More precisely, the function $\mu$ describes light intensity and is sign changing due to the fact that the light is applied to the sample
locally, and areas where $\mu< 0$ are interpreted as shadow zones, while areas
where $\mu > 0$ correspond to illuminated zones. On the other hand the map $f$ describes the electric field induced by the light. Finally, the intensity
of the applied laser light is represented by the parameter $a$. The relevant solution of \eqref{sing} to model the orientation of the molecules in the liquid crystal sample, is a minimizer $v_\epsilon\in H^1(\R^n;\R^m)$ of the energy functional associated to equation \eqref{sing}:
\begin{equation}
\label{funct00}
E(u)=\int_{\R^n}\left[\frac{\epsilon}{2}|\nabla u|^2-\frac{1}{2\epsilon}\mu(x)|u|^2+\frac{1}{4\epsilon}|u|^4-a(\epsilon)f(x) \cdot u\right],
\end{equation}
where $\cdot$ stands for the inner product in $\R^m$. Assuming that $\lim_{\epsilon\to 0}a(\epsilon)=0$ (cf. \cite[Theorems 1.2 and 1.3 (i)]{Clerc2017}, and \cite[Theorem 1.1 (ii)]{Clerc2018}), we discovered that the minimizers $v_\epsilon$ appropriately rescaled in a neighbourhood of a point where $\mu$ vanishes, converge as $\epsilon \to 0$, to a solution $y$ of \eqref{pain 1}. In addition, $y$ is by construction bounded in the half-spaces $[s_0,\infty)\times \R^{n-1}$, $\forall s_0\in\R$, and \emph{minimal} in the sense that
\begin{equation}\label{minnn}
E_{\mathrm{P_{II}}}(y, \mathrm{supp}\, \phi)\leq E_{\mathrm{P_{II}}}(y+\phi, \mathrm{supp}\, \phi)
\end{equation}
for all $\phi\in C^\infty_0(\R^n;\R^m)$.
To explain formally the relation between (\ref{pain 1}) and the energy $E$, one can see from the expression of $E$, that as $\epsilon\to 0$, the modulus of the minimizer $v_{\epsilon}$ should approach a nonnegative root of the polynomial $-\mu(x)z+z^3=0$, or in other words, $|v_{\epsilon}|\to \sqrt{\mu^+}$ as $\epsilon\to 0$, in some perhaps weak sense. This function called the Thomas-Fermi limit of the minimizer is nonsmooth, so the transition near the set $\mu(x)=0$ has to be mediated somehow via a solution of \eqref{pain 1}.
\subsection{Topological defects}
One of the most interesting phenomenon occuring in the two dimensional model ($n=m=2$) considered in \cite{Clerc2018} is the presence in the liquid crystal sample of a new type of topological defect that we named the \emph{shadow vortex}\footnote{The shadow vortex was discovered experimentally in \cite{clerc2}. Then, its existence was confirmed mathematically in \cite[Theorem 1.2 (ii)]{Clerc2018}. Manipulating light vortices has applications in quantic computation, telecommunications, and astronomy (improvementof images, detection of exoplanets).}. It appears at the boundary of the illuminated region when the intensity of the applied laser light is of order $a=o(\epsilon |\ln \epsilon|)$. The computation of the distance of the shadow vortex from the boundary of the illuminated region is a difficult open problem. In view of \cite[Theorem 1.1 (ii)]{Clerc2018}, the shadow vortex is located at a distance of order $O(\epsilon^{\frac{2}{3}})$ from the boundary, if and only if its local profile is given by a \emph{minimal} solution $y:\R^2\to\R^2$ of \eqref{pain 1} having at least one zero. This explains why the investigation of equation \eqref{pain 1} is crucial to understand the formation mechanism, and the properties of shadow vortices in liquid crystals. We expect that equation \eqref{pain 1} may also be relevant to describe the singularities of other Ginzburg-Landau type models similar to \eqref{funct00}, for instance in the context of Bose-Einstein condensates (cf. \cite{ignat,MR2062641,MR3355003,sourdis0}). 
On the other hand, in the one dimensional model ($n=m=1$) topological defects appear at the boundary of the illuminated region, only when $\lim_{\epsilon\to 0}a(\epsilon)\in (0,\sqrt{2})$.
In that case (cf. \cite{sourdis2}) the local profile of the \emph{shadow kink} is given by a sign changing minimal solution of the \emph{nonhomogeneous} O.D.E. $y''-x y-2y^3-\alpha=0$. We also refer to \cite{troy1} for an alternative constuction of this sign changing solution, and to \cite{Clerc2017} for the existence of a positive minimal solution of the nonhomogeneous Painlev\'e O.D.E.

\subsection{One dimensional solutions}\label{subsec:1dd}
In the one dimensional case ($n=m=1$), \eqref{pain 1} reduces to the second Painlev\'e O.D.E.:
\begin{equation}\label{pain 0}
y''-x y-2y^3=0 , \qquad \forall x\in \R,
\end{equation}
which is known to play an important role in the theory of integrable systems \cite{MR1149378}, random matrices \cite{2006math.ph...3038D, Flaschka1980,2005math.ph...8062C}, Bose-Einstein condensates \cite{MR2062641, MR2772375, MR3355003,sourdis0} and other problems \cite{alikoakos_1,helffer1998,KUDRYASHOV1997397}. It has been extensively studied by Painlev\'e and others since the early 1900's. In particular, the solutions of \eqref{pain 0} satisfying the boundary condition $\lim_{+\infty}y=0$ have been classified in \cite{MR555581}. In view of this result, we could establish \cite[Theorem 1.3 (i)]{Clerc2017} that the Hastings-McLeod solution, denoted in this paper by $h$, is up to sign change, the only minimal solution of \eqref{pain 0} which is bounded at $+\infty$. We recall (cf. \cite{MR555581}) that $h:\R\to\R$ is positive, strictly decreasing ($h' <0$) and such that
\begin{align}\label{asy0}
h(x)&\sim \mathop{Ai}(x), \qquad x\to \infty, \nonumber \\
h(x)&\sim \sqrt{|x|/2}, \qquad x\to -\infty,
\end{align}
where $\mathop{Ai}$ is the Airy function. 
Having a closer look at the potential $H(x,y)=\frac{1}{2}xy^2+\frac{1}{2}y^4$, one can explain formally the asymptotic behaviour of $h$. Indeed, for $x$ fixed, $H$ attains its global minimum equal to $0$ when  $y=0$ and $x\geq 0$, and equal to $-\frac{x^2}{8}$ when $y=\pm \sqrt{|x|/2}$ and $x<0$. Thus, the global minima of $H$ bifurcate from the origin, and the two minimal solutions $\pm h$ of \eqref{pain 0} interpolate these two branches of minima.

\subsection{Scalar solutions}
In higher dimensions (cf. \eqref{pain 1} with $n\geq 2$, $m=1$), the scalar P.D.E.
\begin{equation}\label{scalarpde}
\Delta y -x_1y-2y^3=0,\ x=(x_1,\ldots,x_n)\in\R^n,\ y:\R^n\to\R,
\end{equation}
involves the non autonomous potential $H(x_1,y)=\frac{1}{2} x_1 y^2 +\frac{1}{2} y^4$ which is bistable for every fixed $x_1<0$. We have shown in \cite{panos405}, that \eqref{scalarpde} describes a phase transition model, as the Allen-Cahn equation below:
\begin{equation}\label{ac}
\Delta u=W'(u)=u^3-u, \ u:\R^n\to\R, \ W:\R\to [0,\infty), \ W(u)=\frac{1}{4}(u^2-1)^2.
\end{equation}
For the latter the phase transition connects the two minima $\pm1$ of $W$, while for the former the phase transition connects the two branches $\pm \sqrt{(-x_1)^+/2}$ of minima of $H$ parametrized by $x_1$.
More precisely, there exists (cf. \cite{panos405}) a solution $y:\R^2\to\R$ of (\ref{scalarpde})\footnote{This solution was constructed by taking the limit of odd minimizers $v_\epsilon:\R^2\to\R$ of $E$, in a neighbourhood of an appropriate point of the circle $\{x\in\R^2:\mu(x)=0\}$ (cf. \cite{panos405}, and \cite{panayotis_4} for further relation between shadow domain walls and the extended Painlev\'e equation).} converging as $x_2\to \pm \infty$ and $x_1$ is fixed, to the two minimal solutions $\pm h(x_1)$ of \eqref{pain 0}. We detail below its main properties:

\begin{itemize}
\item[(i)] $y$ is positive in the upper half-plane and odd with respect to $x_2$ i.e. $y(x_1,x_2)=-y(x_1,-x_2)$.
\item[(ii)] $y_{x_2}(x_1,x_2)>0$, $\forall x_1, x_2\in\R$, and $\lim_{l\to\pm\infty} y(x_1,x_2+l)=\pm h(x_1)$ in $C^2_{\mathrm{loc}}(\R^2)$.
\item[(iii)] $y$ is minimal (cf. definition \eqref{minnn})\footnote{The minimality of $y$ for general perturbations was pointed out to us by Sourdis \cite{sour}.}.
\item[(iv)] For every $x_2\in\R$ fixed, let $\tilde y(t_1,t_2):=\frac{\sqrt{2}}{(-\frac{3}{2} t_1)^{\frac{1}{3}}}\, y\big(-(-\frac{3}{2} t_1)^{\frac{2}{3}}, x_2+t_2(-\frac{3}{2} t_1)^{-\frac{1}{3}}\big)$. Then
\begin{equation}\label{scale1}
\lim_{l\to -\infty} \tilde y(t_1+l,t_2)=
\begin{cases}
\tanh(t_2/\sqrt{2})  &\text{when } x_2=0, \\
1  &\text{when } x_2>0, \\
-1  &\text{when } x_2<0,
\end{cases}
\end{equation} for the $C^1_{\mathrm{ loc}}(\R^2)$ convergence.
\item[(v)] $y_{x_1}(x_1,x_2)<0$, $\forall x_1\in\R$, $\forall x_2>0$.
\end{itemize}

In view of (i), (ii) and (iii) above, the solution $y$ plays a similar role that the heteroclinic orbit $\gamma( x)= \tan (x/\sqrt{2})$, of the Allen-Cahn O.D.E. $\gamma''=\gamma^3-\gamma$.
First of all, both solutions $y$ and $\gamma$ are minimal and odd. Next, $y$ connects monotonically along the vertical direction $x_2$, the two minimal solutions $\pm h(x_1)$, in the same way that $\gamma$ connects monotonically the two global minimizers $\pm 1$ of the potential $W$. What's more, the two global minimizers $\pm 1$ of the Allen-Cahn functional $E_{\mathrm{AC}}=\int_\R\big(\frac{1}{2}|u'|^2+W(u)\big)$ have their counterparts in the two minimal solutions $\pm h$ of the Painlev\'{e} equation. While $\gamma$ is a one dimensional object, the solution $y(x_1,x_2)$ is two dimensional, since $x_1$ parametrizes the branches of minima of the potential $H$, and only $x_2$ is involved in the phase transition. The analogy between equations \eqref{scalarpde} and \eqref{ac} also appears in property (iv). Indeed, after rescaling, the solution $y$ converges as $x_1\to-\infty$, to a minimal solution of the Allen-Cahn O.D.E., which is depending on the case either $\gamma$ or $\pm 1$. This is not so surprising because the scalar Painlev\'e P.D.E. \eqref{scalarpde} is obtained by multiplying by $-x_1$, the linear term $u$ in the Allen-Cahn P.D.E. \eqref{ac}, and after rescaling as in \eqref{scale1}, the dependence on $x_1$ disappears as $x_1\to-\infty$.

\subsection{The vector Painlev\'e P.D.E. and the Ginzburg-Landau system}

The scope of the present paper is to investigate the vector equation \eqref{pain 1} (with $m\geq 2$), and construct the first to our knowledge nontrivial examples of solutions. First of all, we shall point out the deep connection of  \eqref{pain 1} with
the Ginzburg-Landau system
\begin{equation}\label{gl}
\Delta u=\nabla W(u)=|u|^2 u-u, \ u:\R^n\to\R^m, \ W:\R^m\to [0,\infty), \ W(u)=\frac{1}{4}(|u|^2-1)^2,
\end{equation}
which has been extensively studied (cf. in particular \cite{bethuel1} and \cite{SS}) due to its application in the theory of superconductors and superfluids.  Actually, the vector Painlev\'e P.D.E. \eqref{pain 1} only differs from \eqref{gl} by the factor $-x_1$ multiplying its linear term. 

On the one hand, the non autonomous potential $H(x_1,y)=\frac{1}{2}x_1|y|^2+\frac{1}{2}|y|^4$ associated to \eqref{pain 1} attains its global minimum equal to $0$ when  $y=0$ and $x_1\geq 0$, and equal to $-\frac{x^2_1}{8}$ when $|y|=\sqrt{|x_1|/2}$ and $x_1<0$. Thus, for every fixed $x_1<0$, the set of minima of $H$ is the sphere $\{y\in\R^m: |y|= \sqrt{|x_1|/2}\}$.
On the other hand, the Ginzburg-Landau potential $W$ is nonnegative, radial, and vanishes only on the unit sphere.
These properties of $W$ imply that \eqref{gl} admits when $n=m\geq 2$, a unique \emph{standard vortex} solution $\eta \in C^\infty(\R^n;\R^n)$ such that
\begin{itemize}
\item[(a)] $\eta$ is $O(n)$-equivariant (i.e. $\eta(gx)=g \eta (x)$, $\forall x\in \R^n$, $\forall g\in O(n)$), or equivalently $\eta(x)=\eta_{\mathrm{rad}}(|x|)\frac{x}{|x|}$, $\forall x\neq 0$, where $\eta_{\mathrm{rad}}$ is a function having an odd extension in $C^\infty(\R)$.
\item[(b)] $\eta_{\mathrm{rad}}$ is increasing, and converges to $1$ at $+\infty$.
\end{itemize}
In addition the solution $\eta$ is \emph{minimal} in the sense that $E_{\mathrm{GL}}(\eta, \mathrm{supp}\, \phi)\leq E_{\mathrm{GL}}(\eta+\phi, \mathrm{supp}\, \phi)$, for all $\phi\in C^\infty_0(\R^n;\R^n)$, where
\begin{equation}\label{funcgl}
E_{\mathrm{GL}}(u, \Omega)=\int_\Omega \left[ \frac{1}{2} |\nabla u|^2 +\frac{1}{4}  (|u|^2 -1)^2\right],\ \Omega\subset \R^n,
\end{equation}
is the Ginzburg-Landau energy functional.
Actually in dimension $n=2$, Mironescu \cite{mironescu} established (cf. also \cite{MR1267609}) that any minimal solution of (\ref{gl}) is either constant of modulus $1$ or
(up to orthogonal transformation in the range and translation in the domain) the standard vortex $\eta$.
In higher dimensions $n=m\geq 3$, the minimality of $\eta$ was proved by Pisante \cite{pisante}, however it is not clear if there exist other nontrivial minimal solutions.

Our purpose in this paper is to construct the analog of the standard vortex solution $\eta$ for the vector Painlev\'e P.D.E. \eqref{pain 1}. Since for every fixed $x_1<0$, the potential $H(x_1,y)$ attains its global minimum on the sphere $\{y\in\R^m: |y|= \sqrt{|x_1|/2}\}$, we should have $m=n-1$ in order to allow the formation of vortices in the hyperplanes $x_1=\mathrm{Const}$. On the other hand, the amplitude of these vortices will depend on the radius $\sqrt{|x_1|/2}$. Thus, the standard vortex solution of \eqref{pain 1} should be a solution $y:\R^n\to\R^{n-1}$, $n\geq 3$, such that
\begin{itemize}
	\item[(a)] $y$ is $O(n-1)$-equivariant with respect to $(x_2,\ldots,x_n)=:z$ (i.e. $y(x_1,gz)=g y(x_1,z)$, $\forall x_1\in \R$, $\forall z \in \R^{n-1}$, $\forall g\in O(n-1)$).
	\item[(b)] $|y(x_1,z)|\approx \sqrt{|x_1|/2}$, as $|z|\to\infty$, and $x_1<0$ is fixed.
\end{itemize}

More precisely, we have:

\begin{theorem}\label{corpain2}
There exists a solution $y\in C^\infty(\R^n;\R^{n-1})$ (with $ n \geq 3$) to
\begin{equation}\label{painhom}
\Delta y-x_1 y-2|y|^2y=0, \text{ with } x=(x_1,\ldots,x_n)\in \R^n,\ y=(y_2,\ldots,y_n):\R^n\to\R^{n-1},
\end{equation}
such that
\begin{itemize}
\item[(i)] Setting $z:=(x_2,\ldots,x_n)$, $e_z:=\frac{z}{|z|}$, and $\sigma:=|z|$, we have $y(x)= y_{\mathrm{rad}}(x_1,\sigma)e_z$, where $y_{\mathrm{rad}}(x_1,\sigma)$ is a function having an odd with respect to $\sigma$ extension in $C^\infty(\R^2;\R)$.
\item[(ii)] In addition, $y_{\mathrm{rad}}(x_1,\sigma)>0$, $\frac{\partial y_{\mathrm{rad}}}{\partial x_1}(x_1,\sigma)<0$, and $\frac{\partial y_{\mathrm{rad}}}{\partial \sigma}(x_1,\sigma)>0$,  $\forall x_1\in\R$, $\forall \sigma>0$.
\item[(iii)] $|y(x)|<h(x_1)$, $\forall x \in \R^{n}$, and $\lim_{l\to\infty}  y_{\mathrm{rad}}(x_1,\sigma+l)=h(x_1)$ in $C^1_{\mathrm{loc}}(\R^2;\R)$,
where $h$ is the Hastings-McLeod solution of (\ref{pain 0}).
\item[(iv)] For every $z=(x_2,\ldots,x_n)\in\R^{n-1}$ fixed, let
\begin{equation}\label{ytilde}
\tilde y(t_1,\ldots,t_n):=\frac{\sqrt{2}}{(-\frac{3}{2} t_1)^{\frac{1}{3}}}\, y\big(-(-\frac{3}{2} t_1)^{\frac{2}{3}}, x_2+t_2(-\frac{3}{2} t_1)^{-\frac{1}{3}}, \ldots, x_n+t_n(-\frac{3}{2} t_1)^{-\frac{1}{3}}    \big).
\end{equation}
Then
\begin{equation}\label{scale2}
\lim_{l\to -\infty} \tilde y(t_1+l,t_2,\ldots, t_n)=
\begin{cases}
\eta(t_2,\ldots,t_n)  &\text{when } z=0, \\
e_z  &\text{when } z\neq 0,
\end{cases}
\end{equation} for the $C^1_{\mathrm{ loc}}(\R^n;\R^{n-1})$ convergence, where $\eta\in C^\infty(\R^{n-1};\R^{n-1})$ is the standard vortex solution of the Ginzburg-Landau system \eqref{gl}.
\end{itemize}
\end{theorem}

Theorem \ref{corpain2} provides a solution $y:\R^n\to\R^{n-1}$ having in every hyperplane $x_1=\mathrm{Const.}$, a profile similar to $\eta$. Indeed, for fixed $x_1$, the $O(n-1)$-equivariant map $z\mapsto y(x_1,z)$ only vanishes at $z=0$, and its modulus increases as $|z|$ increases. The amplitude of these vortices is determined by the Hastings-McLeod solution $h$ evaluated at $x_1$. As we mentioned in subsection \ref{subsec:1dd}, $h$ interpolates smoothly the function $x_1\mapsto \sqrt{(-x_1)^+/2}$ describing the radius of the sphere where the potential $H(x_1,y)$ attains its global minimum. On the other hand, for every fixed $z\neq 0$, the map $x_1\mapsto |y(x_1,z)|$ is decreasing, like $h$.
Finally, property (iv) shows that after rescaling, the solution $y$ converges as $x_1\to-\infty$, to a solution of the Ginzburg-Landau system \eqref{gl}, which is depending on the case either $\eta$ or a constant of modulus $1$
(compare with \eqref{scale1}).
Actually, it is proved in Lemma \ref{ass}, that the rescaling \eqref{scale2} applied to any solution of \eqref{pain 1} satisfying the bound \eqref{boundass}, provides after passing to limit, a solution of \eqref{gl}.

\subsection{Minimal solutions of the vector Painlev\'e P.D.E}
Despite the deep connection of the vector Painlev\'{e} P.D.E. with the Ginzburg-Landau system, we are not aware if the structure of minimal solutions of \eqref{pain 1} exactly mirrors that of \eqref{gl} at least in low dimensions.
Although by construction the solution $y:\R^n\to\R^{n-1}$ provided by Theorem \ref{corpain2} is only minimal for $O(n-1)$-equivariant perturbations, we expect that $y$ is actually minimal for general perturbations, as the standard vortex $\eta:\R^{n-1}\to\R^{n-1}$ of the Ginzburg-Landau system. While $\eta$ is defined in $\R^{n-1}$, the solution $y$ depends on $n$ variables, since $x_1$ parametrizes the set of minima of the potential $H$, and only $x_2,\ldots,x_n$ are involved in the vortex formation mechanism. 

In the one dimensional case, one can easily see (cf. Lemma \ref{cara}) that the only minimal solutions $y:\R\to\R^m$ of \eqref{pain 1} bounded at $+\infty$, are the maps $\R\ni x_1\mapsto h(x_1)\n$, with $\n\in\R^m$ a unit vector. These maps have their counterparts in the constant solutions of modulus $1$ of \eqref{gl} (which are also the global minimizers of $E_{\mathrm{GL}}$).
On the other hand, in dimension two, the existence of nontrivial minimal solutions $y:\R^2\to\R^m$ of \eqref{pain 1} is not clear, since the Painlev\'e system with $n=2$ and $m\geq 2$, is related to the O.D.E. system $u''=|u|^2-u$, $u:\R\to\R^m$, which has only constant minimal solutions (cf. \cite[Remark 3.5.]{antonop}, and also \cite{farina} for nonexistence results of minimal solutions of the Ginzburg-Landau system in higher dimensions).

As far as the liquid crystal model \eqref{sing} is concerned, with $n=m=2$ as in \cite{Clerc2018}, the nonexistence of a minimal solution $y:\R^2\to\R^2$ of \eqref{pain 1} having an isolated zero, would imply the following two important results. Firstly, that the profile of the shadow vortex (appearing when the intensity of the applied laser light is of order $o(\epsilon |\ln \epsilon|)$, cf. \cite[Theorem 1.2 (ii)]{Clerc2018}), is given by the Ginzburg-Landau system \eqref{gl}. Secondly, that it is located at a distance $d \gg \epsilon^{2/3}$ from the boundary of the illuminated region.


\section{General results for solutions $y:\R^n\to\R^m$ of \eqref{pain 1}}\label{sec:sec3}


In this section we collect some general results holding for every solution $y:\R^n\to\R^m$ of \eqref{pain 1}. These results may be 
useful to construct different types of solutions of \eqref{pain 1}. They will be used and particularized in the proof of Theorem \ref{corpain2}, in Section \ref{sec:sec4}. 

In the next two lemmas, we shall first examine the asymptotic behavior of solutions of \eqref{pain 1}, as $x_1\to+\infty$, and $x_1\to-\infty$.

\begin{lemma}\label{expcvv}
Let $y:\R^n\to\R^m$ be a solution of \eqref{pain 1} which is bounded in the half-space $\{x_1\geq 1\}$. Then, we have
$|y(x)|=O(e^{-\frac{2}{3}x_1^{3/2}})$, as $x_1\to\infty$ (uniformly in $z=(x_2,\ldots,x_n)$).
\end{lemma}
\begin{proof}
Let $M=e^{\frac{2}{3}}\sup_{x_1\geq 1} |y(x)|$, let $y=(y_1,\ldots,y_m)\in\R^m$, and let $\Omega=\{ x_1>1\}$.
We shall compare $\pm y_i$, $\forall i=1,\ldots,m$, with the function $\psi(x):=M e^{-\frac{2}{3} x_1^{3/2}}$. It is clear that $\Delta \psi \leq x_1\psi\leq (x_1+2|y|^2)\psi$ on $\Omega$, and $\Delta (y_i-\psi) \geq (x_1+2|y|^2) (y_i-\psi)$ on $\Omega$.
Since we have $y_i-\psi\leq 0$ on $\partial \Omega$, it follows from the maximum principle (cf. \cite[Lemma 2.1]{beres}) that $y_i-\psi\leq 0$ holds on $\Omega$. Similarly, we obtain that $-y_i-\psi\leq 0$ holds on $\Omega$.
\end{proof}

\begin{lemma}\label{ass}
Let $y:\R^n\to\R^m$ be a solution of \eqref{pain 1} such that the function 
\begin{equation}\label{boundass}
x\mapsto \frac{|y(x)|}{\sqrt{|x_1|}} \text{ is bounded in the half-space $\{x_1\leq -1\}$},
\end{equation}
and let $$t=(t_1,t_2,\ldots,t_n):=\big(-\frac{2}{3} (-x_1)^{\frac{3}{2}}, (-x_1)^{\frac{1}{2}}r_2, \ldots, (-x_1)^{\frac{1}{2}}r_n,\big), \ \forall x_1\leq -1, \ \forall r:=(r_2,\ldots, r_n)\in\R^{n-1}.$$
Equivalently, setting $\tau:=(t_2,\ldots, t_n)\in\R^{n-1}$, we have $(x_1,r)=\big(-(-\frac{3}{2} t_1)^{\frac{2}{3}}, (-\frac{3}{2} t_1)^{-\frac{1}{3}}\tau\big)$.
Next, define $\tilde y(t_1,\tau):=\frac{\sqrt{2}}{(-\frac{3}{2} t_1)^{\frac{1}{3}}}\, y(x_1, r+z)$, for every $z:=(x_2,\ldots,x_n)\in\R^{n-1}$ \emph{fixed}, or equivalently
\begin{equation}\label{eqder}
y(x_1, r+z)=\frac{(-x_1)^{\frac{1}{2}}}{\sqrt{2}}\tilde y(t_1,\tau).
\end{equation}
Then, up to subsequence,
\begin{equation}\label{scale2p}
\lim_{l\to -\infty} \tilde y(t_1+l,t_2,\ldots, t_n)=u(t_1,t_2,\ldots,t_n),
\end{equation}
for the $C^1_{\mathrm{ loc}}(\R^n;\R^{m})$ convergence, where $u\in C^\infty(\R^n;\R^m)$ solves $\Delta u=|u|^2u-u$ in $\R^n$. In addition, if $y$ is a minimal solution of \eqref{pain 1}, then $u$ is also minimal.
\end{lemma}
\begin{proof}

We are going to show that $\tilde y(t_1,\ldots,t_n)$ is uniformly bounded up to the second derivatives, when $\tau=(t_2,\ldots,t_n)$ belongs to a compact set and $t_1\to-\infty$.
By differentiating \eqref{eqder} with respect to $x_1$ and $r$ we obtain
\begin{subequations}\label{eqderder}
\begin{equation}\label{eqder1}
\sqrt{2}y_{x_i}(x_1, r+z)=(-x_1)\tilde y_{t_i}(t_1,\tau), \ \forall i=2,\ldots,n.
\end{equation}
\begin{equation}\label{eqder2}
\sqrt{2}y_{x_ix_j}(x_1, r+z)=(-x_1)^{\frac{3}{2}}\tilde y_{t_it_j}(t_1,\tau),\ \forall i,j=2,\ldots,n.
\end{equation}
\begin{equation}\label{eqder3}
 \sqrt{2}y_{x_1}(x_1, r+z)=-\frac{(-x_1)^{-\frac{1}{2}}}{2}\tilde y(t_1,\tau)+(-x_1) \tilde y_{t_1}(t_1,\tau)-\sum_{i=2}^n\frac{r_i}{2}\tilde y_{t_i}(t_1,\tau).
\end{equation}
\begin{equation}\label{eqder4}
 \sqrt{2}y_{x_1x_j}=-\tilde y_{t_j}+(-x_1)^{\frac{3}{2}}\tilde y_{t_1t_j}-(-x_1)^{\frac{1}{2}}\sum_{i=2}^n\frac{r_i}{2}\tilde y_{t_it_j}, \ \forall j=2,\ldots,n.
\end{equation}
\begin{equation}\label{eqder5}
 \sqrt{2}y_{x_1x_1}=-\frac{(-x_1)^{-\frac{3}{2}}}{4}\tilde y-\frac{3}{2}\tilde y_{t_1} +\frac{(-x_1)^{-1}}{4} \sum_{i=2}^n r_i \tilde y_{t_i}
+(-x_1)^{\frac{3}{2}}\tilde y_{t_1t_1}-(-x_1)^{\frac{1}{2}}\sum_{i=2}^n r_i \tilde y_{t_1t_i}+\frac{(-x_1)^{-\frac{1}{2}}}{4}\sum_{i,j=2}^nr_ir_j \tilde y_{t_i,t_j}.
\end{equation}
\end{subequations}
Since by assumption $y$ satisfies $|y(x)|=O(|-x_1|^{\frac{1}{2}})$ as $x_1\to-\infty$ (i.e. $\tilde y$ is bounded), we obtain by \eqref{pain 1} and standard elliptic estimates \cite[\S 3.4 p. 37 ]{1987130} that
\begin{equation}\label{eqder6}
\text{$|D y(x)|=O(|-x_1|^{\frac{3}{2}})$ and $|D^2 y(x)|=O(|-x_1|^{\frac{5}{2}})$, as $x_1\to-\infty$.}
\end{equation}
From \eqref{eqder6} and \eqref{eqderder} it follows that
\begin{equation}\label{eqder7}
\text{$|\nabla \tilde y(t)|=O(|-x_1|^{\frac{1}{2}})$ and $|D^2 \tilde y(t)|=O(|-x_1|)$, as $x_1\to-\infty$,}
\end{equation}
provided that $t=(t_1,\tau)\in \Sigma_{\alpha,\beta}:=\{ t\in \R^n: t_1\leq \alpha, \, |\tau|\leq \beta (-\frac{3}{2} t_1)^{\frac{1}{3}}\}$, where $\alpha<0$ and $\beta>0$ are arbitrary constants. In particular, we have $\sqrt{2}\Delta y(x_1,r+z)=(-x_1)^{\frac{3}{2}}\Delta \tilde y(t)+O(|-x_1|^{\frac{3}{2}})$, for $t\in \Sigma_{\alpha,\beta}$. On the other hand it is clear by \eqref{pain 1} that
$\sqrt{2}\Delta y(x_1,r+z)=(-x_1)^{\frac{3}{2}}(|\tilde y(t)|^2-1)\tilde y(t)$, thus
\begin{equation}\label{eqder8}
\text{$|\Delta \tilde y(t)|$ and $|\nabla\tilde y(t)|$ are bounded, $\forall t\in \Sigma_{\alpha,\beta}$.}
\end{equation}
Similarly, by differentiating once more equations \eqref{eqderder} with respect to $x_1$ and $r$, one can show that
\begin{equation}\label{eqder9}
\text{$|D^2\tilde y(t)|$ is bounded, $\forall t\in \Sigma_{\alpha,\beta}$.}
\end{equation}
Next, we apply the theorem of Ascoli to the sequence $\tilde y(t_1+ l,t_2,\ldots t_n)$ as $ l\to-\infty$. Up to a subsequence $l_k\to -\infty$, we obtain via a diagonal argument, the convergence in $C^1_{\mathrm{ loc}}(\R^n;\R^m)$ of $\tilde y_k(t_1,t_2,\ldots,t_n):=\tilde y(t_1+l_k,t_2,\ldots,t_n)$ to a bounded function $ u(t_1,t_2,\ldots,t_n)$ that we are going to determine. 

Let $(e_1,\ldots,e_n)$ be the canonical basis of $\R^n$, and let $\tilde \phi(t_1,\ldots,t_n)\in C^\infty_0(\R^n;\R^m)$ be a test function such that
$\tilde S:=\supp \tilde \phi \subset\{(t_1,\ldots,t_n): c-d\leq t_1\leq c\} $, for some constants $c \in\R$ and $d>0$.
Given $l\in\R$, we consider the translated functions $\tilde \phi^{-l}(t_1,\ldots,t_n):=\tilde \phi(t_1-l,t_2,\ldots,t_n)$, and
$\tilde y^l(t_1,\ldots,t_n):=\tilde y(t_1+l,t_2,\ldots,t_n)$.
Note that $\tilde S^{l}:=\supp \tilde \phi^{-l}=\tilde S+l e_1$, and
$\supp\tilde  \phi^{-l} \subset\{(t_1,\ldots,t_n):  t_1<-1\} $ when $l<-1-c$. Thus, for $l<1-c$, we can define $\phi^{-l}\in C^\infty_0(\R^n;\R^m)$ by $\phi^{-l}(x_1, r+z)=\frac{(-x_1)^{\frac{1}{2}}}{\sqrt{2}}\tilde \phi^{-l}(t_1,\tau)$
as in \eqref{eqder}, and we set $S^{l}:=\{(x_1(t_1),r(t_1,\ldots,t_n)+z): \, (t_1,\ldots,t_n)\in \tilde S^{l}\}$.
As a consequence, we have
\begin{equation}\label{weak}
\int_{\R^n}\nabla y\cdot\nabla \phi^{-l}+x_1 y\cdot\phi^{-l}+2|y|^2y\cdot\phi^{-l}=0, \ \forall l<-1-c,
\end{equation}
that becomes after changing variables as in (\ref{eqder}):
\begin{equation}\label{weak2}
\int_{\R^n}\frac{1}{2}\big(-\frac{3}{2}t_1\big)^{\frac{4-n}{3}}[\nabla \tilde y\cdot\nabla \tilde \phi^{-l}+( |\tilde y|^2-1)\tilde y\cdot\tilde \phi^{-l}+A(\tilde y,\tilde \phi^{-l})]=0, \ \forall l<-1-c,
\end{equation}
with
\begin{multline*}
A(\psi,\xi):=-\frac{1}{2}\big(-\frac{3}{2}t_1\big)^{-1}[\psi\cdot\xi_{t_1}+\psi_{t_1}\cdot \xi+\sum_{i=2}^nt_i (\psi_{t_i}\cdot\xi_{t_1}+\psi_{t_1}\cdot \xi_{t_i})]\\
+\frac{1}{4}\big(-\frac{3}{2}t_1\big)^{-2}[\psi+\sum_{i=2}^nt_i \psi_{t_i}][\xi+\sum_{j=2}^nt_j \xi_{t_j}].
\end{multline*}
On the other hand, in view of \eqref{eqder8}, one can see that
\begin{equation}\label{ll1}
(-t_1)^\beta=(-c-l_k)^\beta+O((-c-l_k)^{\beta-1}), \ \forall \beta<1, \ \forall (t_1,\ldots,t_n)\in \tilde S^{l_k},
\end{equation}
\begin{equation*}\label{ll2}
\int_{\R^n}(-t_1)^{\frac{4-n}{3}}A(\tilde y,\tilde \phi^{-l_k})=O((-c-l_k)^{\frac{1-n}{3}}), 
\end{equation*}
and
\begin{equation*}\label{ll3}
\int_{\R^n}(-t_1)^{\frac{4-n}{3}}[\nabla \tilde y\cdot\nabla \tilde \phi^{-l_k}+( |\tilde y|^2-1)\tilde y\cdot\tilde \phi^{-l_k}]=(-c-l_k)^{\frac{4-n}{3}}\int_{\R^n}[\nabla \tilde y\cdot\nabla \tilde \phi^{-l_k}+( |\tilde y|^2-1)\tilde y\cdot\tilde \phi^{-l_k}]+O((-c-l_k)^{\frac{1-n}{3}}).
\end{equation*}
Gathering the previous results, it follows from \eqref{weak2} that
\begin{equation*}\label{ll4}
(-c-l_k)^{\frac{4-n}{3}}\int_{\R^n}[\nabla \tilde y\cdot\nabla \tilde \phi^{-l_k}+( |\tilde y|^2-1)\tilde y\cdot\tilde \phi^{-l_k}]+O((-c-l_k)^{\frac{1-n}{3}})=0.
\end{equation*}
Finally, since
\begin{align*}\label{ll4}
\lim_{k\to\infty}\int_{\R^n}[\nabla \tilde y\cdot\nabla \tilde \phi^{-l_k}+( |\tilde y|^2-1)\tilde y\cdot\tilde \phi^{-l_k}]&=
\lim_{k\to\infty}\int_{\R^n}[\nabla \tilde y_k \cdot\nabla \tilde \phi+( |\tilde y_k|^2-1)\tilde y\cdot\tilde \phi]\\
&=\int_{\R^n}[\nabla u\cdot\nabla \tilde \phi+( |\tilde u|^2-1)u\cdot\tilde \phi],
\end{align*}
we deduce that $\int_{\R^n}[\nabla u\cdot\nabla \tilde \phi+( |\tilde u|^2-1)u\cdot\tilde \phi]$, $\forall \tilde \phi \in C^\infty_0(\R^n;\R^m)$, i.e. $u$ solves $\Delta u=|u|^2u-u$ in $\R^n$.

Similarly, if $y$ is a minimal solution of \eqref{pain 1}, we have
\begin{equation}\label{weakb}
E_{\mathrm{P_{II}}}(y, S^l)\leq E_{\mathrm{P_{II}}}(y+ \phi^{-l},S^l), \ \forall l<-1-c,
\end{equation}
that becomes after changing variables as in (\ref{eqder}):
\begin{equation}\label{weak2b}
\int_{\tilde S^l}\frac{1}{2}\big(-\frac{3}{2}t_1\big)^{\frac{4-n}{3}}[B( \tilde y)+C(\tilde y)]\leq  \int_{\tilde S^l}\frac{1}{2}\big(-\frac{3}{2}t_1\big)^{\frac{4-n}{3}}[B(\tilde y+\tilde \phi^{-l})+C(\tilde y+\tilde \phi^{-l})], \ \forall l<-1-c,
\end{equation}
with
\begin{equation}\label{bdef}
B(\psi)=[\frac{1}{2}|\nabla \psi|^2-\frac{|\psi|^2}{2}+\frac{|\psi|^4}{4}],
\end{equation}
and
\begin{equation}\label{cdef}
C(\psi):=-\frac{1}{2}\big(-\frac{3}{2}t_1\big)^{-1}[\psi\cdot\psi_{t_1}+\sum_{i=2}^n t_i (\psi_{t_i}\cdot\psi_{t_1})]+\frac{1}{8}\big(-\frac{3}{2}t_1\big)^{-2}|\psi +\sum_{i=2}^nt_i \psi_{t_i}|^2.
\end{equation}
As previously \eqref{ll1} holds, and one can see that
\begin{equation*}\label{ll2}
\int_{\tilde S^{l_k}}(-t_1)^{\frac{4-n}{3}}C(\tilde y)=O((-c-l_k)^{\frac{1-n}{3}}),\ \int_{\tilde S^{l_k}}(-t_1)^{\frac{4-n}{3}}C(\tilde y+\tilde \phi^{-l_k})=O((-c-l_k)^{\frac{1-n}{3}}),
\end{equation*}
\begin{equation*}\label{ll3}
\int_{\tilde S^{l_k}}(-t_1)^{\frac{4-n}{3}}B(\tilde y)=(-c-l_k)^{\frac{4-n}{3}}\int_{\tilde S^{l_k}}B(\tilde y)+O((-c-l_k)^{\frac{1-n}{3}}),
\end{equation*}
and
\begin{equation*}\label{ll3}
\int_{\tilde S^{l_k}}(-t_1)^{\frac{4-n}{3}}B(\tilde y+\tilde \phi^{-l_k})=(-c-l_k)^{\frac{4-n}{3}}\int_{\tilde S^{l_k}}B(\tilde y+\tilde \phi^{-l_k})+O((-c-l_k)^{\frac{1-n}{3}}).
\end{equation*}
Gathering the previous results, it follows from \eqref{weak2b} that
\begin{equation*}\label{ll4}
(-c-l_k)^{\frac{4-n}{3}}E_{\mathrm{GL}}(\tilde y, \mathrm{supp}\, \tilde \phi^{-l_k})\leq (-c-l_k)^{\frac{4-n}{3}}E_{\mathrm{GL}}(\tilde y+\tilde \phi^{-l_k}, \mathrm{supp}\, \tilde \phi^{-l_k})+O((-c-l_k)^{\frac{1-n}{3}}),
\end{equation*}
or equivalently
\begin{equation*}\label{ll5}
E_{\mathrm{GL}}(\tilde y_k, \mathrm{supp}\, \tilde \phi)
\leq
E_{\mathrm{GL}}(\tilde y_k+\tilde\phi, \mathrm{supp}\, \tilde \phi)
+O((-c-l_k)^{-1}),
\end{equation*}
Finally, passing to the limit, we obtain
\begin{equation*}\label{ll4}
E_{\mathrm{GL}}(u, \mathrm{supp}\, \tilde \phi)\leq E_{\mathrm{GL}}(u+\tilde \phi, \mathrm{supp}\, \tilde \phi),
\end{equation*}
i.e. $u$ is a minimal solution of \eqref{gl}.

\end{proof}

Next, we establish some properties of minimal solutions of \eqref{pain 1}. More precisely, we prove in Lemma \ref{lemzero} that a minimal solution is not identically zero, while in Lemma \ref{cara} (resp. Lemma \ref{lema}) we characterize the one dimensional minimal solutions of \eqref{pain 1} (resp. the nonnegative minimal solution of the scalar equation \eqref{scalarpde}).

\begin{lemma}\label{lemzero}
Let $y:\R^n\to\R^m$ be a minimal solution of \eqref{pain 1}. Then, $y$ is not identically $0$.
\end{lemma}
\begin{proof}
The minimality of $y$ implies that the second variation of the energy $E_{\mathrm{P_{II}}}$ is nonnegative:
\begin{equation}\label{secondvabis}
\int_{\R^n} (|\nabla \phi(x)|^2+(2|y(x)|^2+4(y\cdot\phi)^2+x_1)|\phi(x)|^2)\dd x \geq 0,\quad  \forall \phi\in C^1_0(\R^n;\R^m). 
\end{equation}
Clearly \eqref{secondvabis} does not hold when $y\equiv 0$, if we take $\phi(x_1,z)=\phi_0(x_1+l,z)$, with $l\to\infty$, and $\phi_0\in C^1_0(\R^n;\R^m)$, such that $\phi_0 \not\equiv 0$.
\end{proof}

\begin{lemma}\label{cara}
Let $y:\R\to\R^m$ be a solution of \eqref{pain 1} which is bounded at $+\infty$. Then, $y$ is minimal iff $y(x)= h(x)\n$, for some unit vector $\n\in \R^m$.
\end{lemma}
\begin{proof}
Assume that $y$ is minimal. Since $y\not\equiv 0$ (cf. Lemma \ref{lemzero}), let $x_0\in\R$ be such that $y(x_0)\neq 0$, and let $\n_0:=\frac{y(x_0)}{|y(x_0)|}$. Next, we consider the competitor map
\begin{equation}
\xi(x)=\begin{cases}
y(x)  &\text{when } x\leq x_0, \\
|y(x)|\n_0  &\text{when } x\geq x_0.
\end{cases}
\end{equation}
It is clear that $|\xi'|\leq |y'|$ holds on $\R$, and that $E_{\mathrm{P_{II}}}(\xi,[a,b])\leq E_{\mathrm{P_{II}}}(y,[a,b])$, $\forall a<b$. Assume by contradiction that $E_{\mathrm{P_{II}}}(\xi,[x_0,\infty))+2\epsilon< E_{\mathrm{P_{II}}}(y,[x_0,\infty))$, for some $\epsilon>0$. This implies that for $b>b_0$ large enough we have $E_{\mathrm{P_{II}}}(\xi,[x_0,b])+\epsilon< E_{\mathrm{P_{II}}}(y,[x_0,b])$. Setting
\begin{equation}
\psi(x)=\begin{cases}
\xi(x)  &\text{when } x\in [x_0,b] \\
(y(b+1)-|y(b)|\n_0 )(x-b)+|y(b)|\n_0 &\text{when } x\in [b,b+1],
\end{cases}
\end{equation}
it follows from Lemma \ref{expcvv} that $E_{\mathrm{P_{II}}}(\psi,[b,b+1])<\epsilon$ provided that $b\geq b_1\geq b_0$ is large enough. Thus we deduce that $E_{\mathrm{P_{II}}}(\psi,[x_0,b_1+1])< E_{\mathrm{P_{II}}}(y,[x_0,b_1+1])$, with $y(x_0)=\psi(x_0)$, and $y(b_1+1)=\psi(b_1+1)$, in contradiction with the minimality of $y$. This proves that $E_{\mathrm{P_{II}}}(\xi,[x_0,\infty))= E_{\mathrm{P_{II}}}(y,[x_0,\infty))$. In particular, we have  $E_{\mathrm{P_{II}}}(\xi,[x_0,b])=E_{\mathrm{P_{II}}}(y,[x_0,b])$, on an interval $[x_0,b]$ where $y\neq 0$, and  setting $\n(x):=\frac{y(x)}{|y(x)|}$ on $[x_0,b]$, we obtain  $|\xi'|^2=|y'|^2=|\xi'|^2+|\xi|^2|\n'|^2$ on $[x_0,b]$.
As a consequence, $\n'\equiv 0$, and $y(x)=|y(x)|\n_0$ hold on $[x_0,b]$. Finally, for any unit vector $\nu$ perpendicular to $\n_0$, let $\chi(x)=y(x)\cdot \nu$. In view of \eqref{pain 1}, $\chi$ solves $\chi''=x\chi+2|y|^2\chi$ on $\R$, and since  $\chi\equiv 0$ on $[x_0,b]$, we conclude by the uniqueness result for O.D.E. that $\chi\equiv 0$ on $\R$. This proves that $y(x)=(y(x)\cdot\n_0)\n_0$ holds on $\R$. Therefore, it follows from the characterization of minimal solutions of \eqref{pain 0} established in \cite[Theorem 1.3 (i)]{Clerc2017}, that $y(x)=h(x)\n_0$, $\forall x\in\R$.

Conversely, let $y(x)=h(x)\n$, for some unit vector $\n\in\R^m$, and let us check that $y$ is minimal. 
For every test function $\phi\in H^1_0([a,b];\R^m)$, let $\xi(x):=|y(x)+\phi(x)|\n$. Since $|\xi'|\leq |y'+\phi'|$ holds on $\R$, we have $E_{\mathrm{P_{II}}}(\xi,[a,b])\leq E_{\mathrm{P_{II}}}(y+\phi,[a,b])$. On the other hand, it follows from the minimality of $h:\R\to\R$ that 
$E_{\mathrm{P_{II}}}(y,[a,b])\leq E_{\mathrm{P_{II}}}(\xi,[a,b])$. This establishes that $y$ is minimal.

\end{proof}
\begin{lemma}\label{lema}\footnote{A similar result holds for the nonnegative minimal solutions of the Allen-Cahn equation \eqref{ac} (cf. for instance \cite[Corollary 5.2]{book} which also applies in the vector case).} 
Let $y:\R^n\to\R$ be a nonnegative minimal solution of \eqref{scalarpde} such that $y(x)\leq h(x_1)$, $\forall x\in\R^n$. Then $y(x_1,\ldots,x_n)=h(x_1)$, $\forall x\in \R^n$.
\end{lemma}
\begin{proof}
First of all we show that $y>0$. Indeed, if $y(p)=0$ for some $p\in\R^n$, then the maximum principle implies that $y \equiv 0$, and this is ruled out by Lemma \ref{lemzero}. Let $B_R\subset \R^n$ be the open ball of radius $R$, centered at the origin, and let $y_R$ be a minimizer of $E_{\mathrm{P_{II}}}$ in $H_0^{1}(B_R;\R)$ (cf. Lemma \ref{lem1} for the existence of minimizers). We notice that $|y_R|$ is also a minimizer of $E_{\mathrm{P_{II}}}$ in $H_0^{1}(B_R;\R)$, thus we can choose $y_R$ such that $y_R\geq 0$. It is clear that $y_R\in C^\infty(\overline B_R;\R)$ is a smooth solution of \eqref{scalarpde} in $B_R$. To establish Lemma \ref{lema}, we shall compare $y$ with $y_R$. Our claim is that $y_R\leq y$, $\forall R$. Indeed, assume by contradiction that $S:=\{x\in B_R: y_R>y \}\neq \emptyset$, and let $\phi:=\max(y_R-y,0)\in H_0^{1}(B_R;\R)$. On the one hand, we have by the minimality of $y$ that
\begin{equation*}\label{mm1}
 E_{\mathrm{P_{II}}}(y,B_R)\leq E_{\mathrm{P_{II}}}(y+\phi,B_R)=E_{\mathrm{P_{II}}}(y_R,S)
+E_{\mathrm{P_{II}}}(y,B_R\setminus S),
\end{equation*}
thus $ E_{\mathrm{P_{II}}}(y,S)\leq E_{\mathrm{P_{II}}}(y_R,S)$. On the other hand, by definition of the minimizer $y_R$, we obtain
\begin{equation*}\label{mm2}
E_{\mathrm{P_{II}}}(y_R,B_R)\leq E_{\mathrm{P_{II}}}(y_R-\phi,B_R)=E_{\mathrm{P_{II}}}(y,S)
+E_{\mathrm{P_{II}}}(y_R,B_R\setminus S),
\end{equation*}
that is, $E_{\mathrm{P_{II}}}(y_R,S)\leq E_{\mathrm{P_{II}}}(y,S)$. As a consequence, 
$E_{\mathrm{P_{II}}}(y_R,S)= E_{\mathrm{P_{II}}}(y,S)$, and the function $ \tilde y_R(x):=\min 
(y_R,y)$ is another minimizer of $E_{\mathrm{P_{II}}}$ in $H_0^{1}(B_R;\R)$. In particular $\tilde y_R$ is a $C^\infty(\overline B_R;\R)$ smooth solution of \eqref{scalarpde} in $B_R$. Finally, since $y_R$ and $\tilde y_R$ coincide on an open ball $B\subset \{x\in B_R:y_R(x)<y(x)\}\neq \emptyset$, it follows by unique continuation that $y_R\equiv \tilde y_R$, which is a contradiction.

At this stage, we are going to prove that $y(x_1,\ldots,x_n)=h(x_1)$, $\forall x\in \R^n$ by induction on the dimension $n$. For $n=1$ the statement is true, since $h$ is the only nonnegative minimal solution of O.D.E. \eqref{pain 0}, which is bounded at $+\infty$ (cf. \cite[Theorem 1.3 (i)]{Clerc2017}). Now, assume that $n\geq 2$. We will first establish that $y_R>0$ on $B_R$, provided that $R$ is large enough. As we mentioned before, the existence of $p\in B_R$ such that $y_R(p)=0$ implies that $y_R\equiv 0$. On the other hand, one can see that $0$ is not a minimizer of $E_{\mathrm{P_{II}}}$ in $H_0^{1}(B_R;\R)$, when $R$ is large enough.
Indeed, by taking $\phi(x_1,z)=\phi_0(x_1+l,z)$, with $z=(x_2,\ldots,x_n)$, $l>0$, and $\phi_0\in C^
\infty_0(B_1;\R)$ such that $\phi_0 \not\equiv 0$, we have $\phi \in H_0^{1}(B_{l+1};\R)$, and 
$E_{\mathrm{P_{II}}}(\phi,B_{l+1})<0$, provided that $l$ is large enough. Our claim, is that $\frac{\partial y_R}{\partial x_n}(x)<0$ (resp. $\frac{\partial y_R}{\partial x_n}(x)>0$) provided that $x\in B_R$ (with $R$ large enough) and $x_n>0$ (resp. $x\in B_R$ and $x_n>0$). This follows from the moving plane method applied to $\psi_\lambda(x)=y_R(x)-y_R(x_1, x_2, \ldots,x_{n-1},2\lambda-x_n)$ in the domain $B_{R,\lambda}:=\{x\in B_R:x_n>\lambda\}$, for every $\lambda\in (0,R)$, since $\psi_\lambda$ satisfies 
$\Delta\psi_\lambda(x)=c(x)\psi_\lambda(x)$ in $B_{R,\lambda}$ with $c(x)=x_1+2(y_R^2(x)+y_R^2(x_1,\ldots, x_{n-1},2\lambda-x_n)+y_R(x)y_R(x_1,\ldots,x_{n-1},2\lambda-x_n))$. We refer to \cite[section 9.5.2.]{evans} for more details. The bound $y_R(x)\leq y(x)\leq h(x_1)$ implies that for every $L>0$ fixed, the functions $y_R$, with $R>L+1$, are uniformly bounded in $B_L$, up to the second derivatives. Therefore, by applying the theorem of Ascoli to $y_R$, via a diagonal argument, we can see that (up to subsequence) $y_R$ converges in $C^2_{\mathrm{loc}}(\R^n;\R)$ to a solution  $y_\infty\in C^\infty(\R^n;\R)$ of \eqref{painhom}. Furthermore, $y_\infty$ is by construction minimal, and satisfies $0\leq y_\infty\leq y$ in $\R^n$, as well as 
$\frac{\partial y_\infty}{\partial x_n}(x)\leq 0$ (resp. $\frac{\partial y_\infty}{\partial x_n}(x)\geq 0$) provided that $x_n\geq 0$ (resp. $x_n\leq 0$). In view of this monotonicity property, a second application of the theorem of Ascoli to the sequence $\tilde y_l(x):=y_\infty(x_1,\ldots,x_{n-1},x_n+l)$, shows that $\lim_{l\to\pm\infty}\tilde y_l(x)=Y^\pm(x_1,\ldots,x_{n-1})$, where $Y^\pm:\R^{n-1}\to\R$ is a nonnegative minimal solution of \eqref{scalarpde}. It is clear that $Y^\pm(x)\leq h(x_1)$, thus our induction hypothesis implies that $Y^\pm(x_1,\ldots,x_{n-1})=h(x_1)$. Finally, since $\lim_{l\to\pm\infty} y_\infty(x_1,\ldots,x_{n-1},x_n+l)=h(x_1)$, $\forall x\in\R^n$, and the function $x_n\mapsto y_\infty(x_1,\ldots,x_n)$ is decreasing on $(0,\infty)$ (resp. increasing on $(-\infty,0)$) for every $(x_1,\ldots,x_{n-1})$ fixed, we deduce that $y_\infty(x_1,\ldots,x_n)=h(x_1)\leq y(x)$, and $y(x_1,\ldots,x_n)=h(x_1)$.
\end{proof}

\section{Proof of Theorem \ref{corpain2}}\label{sec:sec4}
In the next lemma we show the existence of an $O(n-1)$-equivariant solution $y$ of \eqref{painhom}. We first construct in every ball $B_R\subset \R^n$ of radius $R$, an $O(n-1)$-equivariant minimizer $y_R$ of $E_{\mathrm{P_{II}}}$. Then, by passing to the limit as $R\to\infty$, we obtain the solution $y$. By construction $y$ vanishes only on the $x_1$ coordinate axis, and $|y(x)|$ is bounded by $h(x_1)$.
\begin{lemma}\label{lem1}
There exists a solution $y\in C^\infty(\R^n;\R^{n-1})$ (with $ n \geq 3$) to \eqref{painhom} such that
\begin{itemize}
\item[(i)] $y$ is $O(n-1)$-equivariant with respect to $z:=(x_2,\ldots,x_n)$, i.e.
\begin{equation}\label{equivv}
y(x_1,gz)=g y(x_1,z),\  \forall x_1\in \R, \ \forall z \in \R^{n-1}, \ \forall g\in O(n-1).
\end{equation}
Consequently, $y$ can be written as $y(x)= y_{\mathrm{rad}}(x_1,|z|)\frac{z}{|z|}$, where $y_{\mathrm{rad}}(x_1,\sigma)$ is a function having and odd with respect to $\sigma$ extension in $C^\infty(\R^2;\R)$.
\item[(ii)] $y$ is minimal with respect to $O(n-1)$-equivariant perturbations, i.e.
\begin{equation}\label{minnn1}
E_{\mathrm{P_{II}}}(y, \mathrm{supp}\, \phi)\leq E_{\mathrm{P_{II}}}(y+\phi, \mathrm{supp}\, \phi) \text{ for all $\phi\in C^\infty_0(\R^n;\R^{n-1})$ satisfying \eqref{equivv}.}
\end{equation}
\item[(iii)] $y(x_1,z)\cdot z>0$, $\forall x_1\in\R$, $\forall z\neq 0$, and $|y(x)|\leq h(x_1)$, $\forall x \in \R^{n}$.
\item[(iv)] \eqref{minnn1} also holds for maps $\psi$ that can be written as $\psi(x_1,z)=\chi(x_1,z)\frac{z}{|z|}$, with $\chi\in C^\infty_0(\R^n;\R)$ and $\supp \chi\subset\R\times (\R^{n-1}\setminus\{0\})$.

\end{itemize}
\end{lemma}
\begin{proof}
Let $B_R\subset \R^n$ (resp. $D_R\subset \R^2$) be the open ball (resp. the open disc) of radius $R$, centered at the origin. 
The existence of a minimizer $y_R$ of $E_{\mathrm{P_{II}}}$ in the class
$$\mathcal A=\{u\in H_0^{1}(B_R;\R^{n-1}): u(x_1,gz)=g u(x_1,z),\text{ for a.e. }  x=(x_1,z)\in B_R,\ \forall g\in O(n-1)\}$$
follows from the direct method. We first show that $\inf\{\,E_{\mathrm{P_{II}}}(u):\ u \in \mathcal A \}>-\infty$. 
To see this, we notice that $ \frac{x_1}{2}|u|^2+\frac{1}{4}|u|^4< 0 \Longleftrightarrow \frac{1}{2}|u|^2<-x_1$, thus $\frac{x_1}{2}|u|^2+\frac{1}{4}|u|^4\geq -x_1^2$.
Now, let $m:=\inf_{\mathcal A} E_{\mathrm{P_{II}}}>-\infty$, and let $u_n$ be a sequence such that $E_{\mathrm{P_{II}}}(u_n) \to m$. According to what precedes, we obtain the bound $\int_{\R^2}\big[\frac{1}{2}| \nabla u_n|^2+\frac{1}{4}|u_n|^4\big]\leq E_{\mathrm{P_{II}}}(u_n)+\int_{B_R} x_1^2 $, hence $\left\|u_n \right\|_{H^{1}(B_R,\R^m)}$ is bounded. 
As a consequence, for a subsequence still called $u_n$, we have $u_n\rightharpoonup y_R$ weakly in $H^1$, as well as $u_n\to y_R$ strongly in $L^4$ and $L^2$. In particular,
$\int_{B_R} |\nabla y_R|^2\leq \liminf_{n \to \infty} \int_{\R^2}  |\nabla u_n|^2$ holds by lower semicontinuity, while $\int_{B_R} |y_R|^4=\lim_{n \to \infty} \int_{B_R}  |u_n|^4$, and $\int_{B_R}x_1|v|^2 = \lim_{n \to \infty} \int_{B_R}x_1 |u_n|^2$ hold due to the strong convergence. Gathering the previous results it is clear that $m= E_{\mathrm{P_{II}}}(y_R,B_R)$. Finally, since $\mathcal A$ is a closed subspace of $ H_0^{1}(B_R;\R^{n-1})$, we deduce that $y_R$ is a minimizer of $ E_{\mathrm{P_{II}}}$ in $\mathcal A$, and a critical point of $ E_{\mathrm{P_{II}}}$ for $O(n-1)$-equivariant perturbations.

Next, since the potential $H$ is invariant with respect to the group $O(n-1)$ (i.e. $H(x_1,gy)=H(x_1,y)$, $\forall x_1\in \R$, $\forall y\in \R^{n-1}$, $\forall g \in O(n-1)$), we obtain in view of \cite{palais}, that $y_R$ is a critical point of $E_{\mathrm{P_{II}}}$ for general $H_0^{1}(B_R;\R^{n-1})$ perturbations. Thus, $y_R\in C^\infty(\overline{B_R};\R^m)$ is a classical solution of \eqref{painhom} in $B_R$, and moreover $y_R$ can be written as $y_R(x)= y_{\mathrm{rad},R}(x_1,|z|)\frac{z}{|z|}$, where $y_{\mathrm{rad},R}(x_1,\sigma)$ is a function having and odd with respect to $\sigma$ extension in $C^\infty(D_R;\R)$.
Computing the energy in cylindrical coordinates, we obtain:
\begin{equation}\label{cyl}
E_{\mathrm{P_{II}}}(y_R, B_R)=A(\SF^{n-2})E_{\mathrm{P_{II},rad}}(y_{\mathrm{rad},R}, D_R\cap\{\sigma>0\})
\end{equation}
with $A(\SF^{n-2})$ the area of the $(n-2)$-dimensional sphere $\SF^{n-2}$, and
\begin{equation}\label{cyl2}
E_{\mathrm{P_{II},rad}}(y_{\mathrm{rad},R},D_R\cap\{\sigma>0\})=\int_{D_R\cap\{\sigma>0\}  } \Big[\frac{1}{2}|\nabla y_{\mathrm{rad},R}|^2 +\frac{n-2}{2}\frac{y_{\mathrm{rad},R}^2}{\sigma^2}+\frac{x_1y_{\mathrm{rad},R}^2}{2}+\frac{y_{\mathrm{rad},R}^4}{2}\Big]\dd x_1\dd \sigma.
\end{equation}
From this expression of $E_{\mathrm{P_{II}}}$, it follows that $\tilde y_R(x):=|y_R(x)|\frac{z}{|z|}$ is another minimizer of $E_{\mathrm{P_{II}}}$ in $\mathcal A$, and also a solution of \eqref{painhom} in $B_R$. Thus, one can choose $y_R$ such that
\begin{equation}\label{psca}
y_R(x_1,z)\cdot z\geq 0,  \ \forall x\in B_R.
\end{equation}
Finally, we are going to show that $|y_R(x)|\leq h(x_1)$ holds on $B_R$. Indeed, assuming by contradiction that $|y_{\mathrm{rad},R}(\tilde x_1,\tilde \sigma)|>h(\tilde x_1)$ holds for some $(\tilde x_1,\tilde \sigma)\in D_R\cap\{\sigma>0\}$, we consider the competitor
$\tilde y_{\mathrm{rad},R}(x_1, \sigma):=\min( y_{\mathrm{rad},R}( x_1,\sigma),h(x_1))$, for $(x_1,\sigma)\in D_R\cap\{\sigma>0\} $, and we notice that the function $$\xi(x_1,\sigma):=\max( y_{\mathrm{rad},R}(x_1, \sigma),h(x_1))-h(x_1)$$ is such that $\supp\xi\subset D_R\cap\{\sigma>0\} $. Thus, since $\R^2\ni (x_1,\sigma)\mapsto h(x_1,\sigma)$ is a minimal solution of \eqref{scalarpde}, we have
\begin{equation}\label{pp11}
                      \int_{\supp\xi}  \Big[\frac{1}{2}|\nabla \tilde y_{\mathrm{rad},R}|^2 +\frac{x_1\tilde y_{\mathrm{rad},R}^2}{2}+\frac{\tilde y_{\mathrm{rad},R}^4}{2}\Big]\dd x_1\dd \sigma  \leq \int_{\supp\xi  } \Big[\frac{1}{2}|\nabla y_{\mathrm{rad},R}|^2 +\frac{x_1y_{\mathrm{rad},R}^2}{2}+\frac{y_{\mathrm{rad},R}^4}{2}\Big]\dd x_1\dd \sigma.
\end{equation}
On the other hand, it is obvious that
\begin{equation}\label{pp12}
                      \int_{\supp\xi} \frac{n-2}{2}\frac{\tilde y_{\mathrm{rad},R}^2}{\sigma^2}\dd x_1\dd \sigma  < \int_{\supp\xi  }\frac{n-2}{2}\frac{y_{\mathrm{rad},R}^2}{\sigma^2}\dd x_1\dd \sigma,
\end{equation}
and
\begin{equation}\label{pp33}
E_{\mathrm{P_{II},rad}}(\tilde y_{\mathrm{rad},R},(D_R\cap\{\sigma>0\})\setminus\supp\xi\})= E_{\mathrm{P_{II},rad}}(y_{\mathrm{rad},R},(D_R\cap\{\sigma>0\})\setminus\supp\xi\}).
\end{equation}
Combining, \eqref{pp11}, \eqref{pp12}, and \eqref{pp33}, we deduce that the map $\tilde y_R(x):= \tilde y_{\mathrm{rad},R}(x_1,|z|)\frac{z}{|z|}$ belonging to $\mathcal A$, satisfies $E_{\mathrm{P_{II}}}(\tilde y_R, B_R)<E_{\mathrm{P_{II}}}(y_R, B_R)$, which is a contradiction. Thus, we have established that $|y_R(x)|\leq h(x_1)$ holds in $B_R$, for every $R>0$. 

The previous bounds imply that for every $L>0$ fixed, the maps $y_R$, with $R>L+1$, are uniformly bounded in $B_L$, up to the second derivatives. Therefore, by applying the theorem of Ascoli to $y_R$, via a diagonal argument, we can see that (up to subsequence) $y_R$ converges in $C^2_{\mathrm{loc}}(\R^n;\R^m)$ to a solution  $y\in C^\infty(\R^n;\R^{n-1})$ of \eqref{painhom} satisfying \eqref{equivv} and \eqref{minnn1}. Again, we point out that $y$ can be written as $y(x)= y_{\mathrm{rad}}(x_1,|z|)\frac{z}{|z|}$, where $y_{\mathrm{rad}}(x_1,\sigma)$ is a function having and odd with respect to $\sigma$ extension in $C^\infty(\R^2;\R)$.

Our next claim is that $y$ cannot be identically zero. Indeed, the minimality of $y$ implies that the second variation of the energy $E_{\mathrm{P_{II}}}$ is nonnegative in the class of $O(n-1)$-equivariant maps:
\begin{equation}\label{secondva}
\int_{\R^n} (|\nabla \phi(x)|^2+(6 |y(x)|^2+x_1)|\phi(x)|^2)\dd x \geq 0,\quad  \forall \phi\in C^1_0(\R^n;\R^m) \text{ satisfying \eqref{equivv}. }
\end{equation}
Clearly \eqref{secondva} does not hold when $y\equiv 0$, if we take $\phi(x_1,z)=\phi_0(x_1+l,z)$, with $l\to\infty$, and $\phi_0$ a smooth, $O(n-1)$-equivariant map, such that $\phi_0 \not\equiv 0$.
Now, let us check that $y(x_1,z)\cdot z>0$ holds for every $ x_1\in\R$ and $ z\neq 0$, or equivalently $x_2>0 \Rightarrow y_2(x)>0$. By construction, we have $y(x_1,z)\cdot z\geq 0$, $\forall x\in \R^n$, and in particular, $y_2\geq 0$ holds in the half-space $\Omega=\{x_2>0\}$. Since $y_2$ satisfies $\Delta y_2=(x_1+2|y|^2)y_2$, the existence of a point $p\in \Omega$ such that $y_2(p)=0$ would imply by the maximum principle that $y_2\equiv 0$, and $y\equiv 0$. This is a contradiction. Thus, we have checked that $y(x_1,z)\cdot z>0$ holds for every $ x_1\in\R$ and $ z\neq 0$, and since it is obvious that $|y(x)|\leq h(x_1)$, $\forall x\in\R^n$, the proof of (iii) is complete.

To prove (iv), we write $y(x)= y_{\mathrm{rad}}(x_1,|z|)\frac{z}{|z|}$, $\phi(x)= \phi_{\mathrm{rad}}(x_1,|z|)\frac{z}{|z|}$, and utilize \eqref{cyl}. One can see that \eqref{minnn1} is equivalent to
\begin{equation}\label{cyl3}
E_{\mathrm{P_{II},rad}}(y_{\mathrm{rad}}, \{|x_1|\leq \alpha, \sigma\in[0, \beta]\})\leq 
E_{\mathrm{P_{II},rad}}(y_{\mathrm{rad}}+\phi_{\mathrm{rad}}, \{|x_1|\leq \alpha, \sigma\in[0, \beta]\}), 
\end{equation}
provided that $\supp \phi_{\mathrm{rad}}\subset [-\alpha,\alpha]\times[-\beta,\beta]$. Now, given $\psi(x_1,z)=\chi(x_1,z)\frac{z}{|z|}$, with $\chi\in C^\infty_0(\R^n;\R)$ and $\supp \chi\subset[-\alpha,\alpha]\times \{z: 0<|z|\leq\beta\}$, we define for every $\nu\in\SF^{n-2}$, the $O(n-1)$-equivariant map $\psi_\nu(x_1,z):=\chi(x_1,|z|\nu)\frac{z}{|z|}$. Setting $\chi_\nu(x_1,\sigma):=\chi(x_1,\sigma\nu)$, we have in view of \eqref{cyl3}:
\begin{equation}\label{cyl4}
E_{\mathrm{P_{II},rad}}(y_{\mathrm{rad}},\{|x_1|\leq \alpha, \sigma\in[0, \beta]\})\leq 
E_{\mathrm{P_{II},rad}}(y_{\mathrm{rad}}+\chi_\nu, \{|x_1|\leq \alpha, \sigma\in[0, \beta]\}), \forall \nu\in\SF^{n-2}.
\end{equation}
We can also check that 
\begin{equation}\label{cyl5}
\frac{1}{2}|\nabla \chi_\nu(x_1,\sigma)|^2 +\frac{n-2}{2}\frac{|\chi_{\nu}(x_1,\sigma)|^2}{\sigma^2}\leq \frac{1}{2}|\nabla \psi(x_1,\sigma\nu)|^2 
\end{equation}
holds for every $|x_1|\leq \alpha$, $\sigma\in[0, \beta]$, $\nu\in\SF^{n-2}$. As a consequence, an integration of \eqref{cyl4} over $\nu\in\SF^{n-2}$ gives:
\begin{equation}\label{cyl6}
E_{\mathrm{P_{II}}}(y, \supp \psi)\leq E_{\mathrm{P_{II}}}(y+\psi, \supp \psi).
\end{equation}

\end{proof}
The uniqueness of the solution $y$ provided by Lemma \ref{lem1} is an open question. In what follows we establish for \emph{any} satisfying the assumptions of Lemma \ref{lem1}, the statements of Theorem \ref{corpain2} hold. We will proceed in few steps.
By particularizing Lemma \ref{ass}, we obtain the limit in \eqref{scale2} in the case where $z\neq 0$:

\begin{lemma}\label{lem44}
Let $y$ be the solution provided by Lemma \ref{lem1}, and consider the rescaled map $\tilde y(t_1,\ldots,t_n)$ as in \eqref{ytilde}, with $z\neq 0$ fixed. Then, $\lim_{l\to\-\infty}\tilde y(t_1+l,t_2,\ldots,t_n)=e_z:=\frac{z}{|z|}$ for the $C^1_{\mathrm{loc}}(\R^n;\R^{n-1})$ convergence.
\end{lemma}
\begin{proof}
We proceed as in the proof of Lemma \ref{ass}. Let $(e_1,\ldots,e_n)$ be the canonical basis of $\R^n$, and let $\tilde \chi(t_1,\ldots,t_n)\in C^\infty_0(\R^n;\R)$ be a test function such that
$\tilde S:=\supp \tilde \phi \subset\{(t_1,\ldots,t_n): c-d\leq t_1\leq c, |(t_2, \ldots ,t_n)|\leq d\}$, for some constants $c \in\R$ and $d>0$.
Given $l\in\R$, we consider the maps $\tilde \chi^{-l}(t_1,\ldots,t_n):=\tilde \chi(t_1-l,t_2,\ldots,t_n)$, and
$\tilde y^l(t_1,\ldots,t_n):=\tilde y(t_1+l,t_2,\ldots,t_n)$.
Note that $\tilde S^{l}:=\supp \tilde \chi^{-l}=\tilde S+l e_1$, and
$\supp\tilde  \chi^{-l} \subset\{(t_1,\ldots,t_n):  t_1<-1\} $ when $l<-1-c$. Furthermore, for $l<l_0:=\min\big(-1-c, -\frac{2}{3}\big(\frac{d}{|z|}\big)^3-c\big)$, we can define $\phi^{-l}\in C^\infty_0(\R^n;\R^{n-1})$ as in \eqref{eqder}, by
$$\phi^{-l}(x_1, r+z)=\frac{(-x_1)^{\frac{1}{2}}}{\sqrt{2}}\tilde \chi^{-l}(t_1,\tau)\frac{r+z}{|r+z|},$$
since we have $S^{l}:=\{(x_1(t_1),r(t_1,\ldots,t_n)+z): \, (t_1,\ldots,t_n)\in \tilde S^{l}\}\subset (-\infty,0)\times (\R^{n-1}\setminus\{0\})$ for $l<l_0$.
As a consequence of Lemma \ref{lem1} (iv), it follows that
\begin{equation}\label{weakaa}
E_{\mathrm{P_{II}}}( y,  S^l)\leq E_{\mathrm{P_{II}}}( y+ \phi^{-l},S^l), \ \forall l<l_0.
\end{equation}
Now, we compute
\begin{subequations}\label{eqderder1}
\begin{equation}\label{eqder1b1}
\sqrt{2}\phi^{-l}_{x_i}(x_1, r+z)=(-x_1)\tilde \chi^{-l}_{t_i}(t_1,\tau)\frac{r+z}{|r+z|}+(-x_1)^{\frac{1}{2}}\tilde \chi^{-l}(t_1,\tau)\Big(\frac{e_i}{|r+z|}-\frac{r_i(r+z)}{|r+z|^3}\Big), \ \forall i=2,\ldots,n,
\end{equation}
\begin{equation}\label{eqder3b1}
 \sqrt{2}\phi^{-l}_{x_1}(x_1, r+z)=-\frac{(-x_1)^{-\frac{1}{2}}}{2}\tilde \chi^{-l}(t_1,\tau)\frac{r+z}{|r+z|}+(-x_1) \tilde \chi^{-l}_{t_1}(t_1,\tau)\frac{r+z}{|r+z|}-\sum_{i=2}^n\frac{r_i}{2}\tilde \chi^{-l}_{t_i}(t_1,\tau)\frac{r+z}{|r+z|},
\end{equation}
\end{subequations}
and we set
\begin{equation}\label{ksi}
\xi(t):=\frac{r+z}{|r+z|}=\frac{(-\frac{3}{2}t_1)^{-\frac{1}{3}}\tau+z}{|(-\frac{3}{2}t_1)^{-\frac{1}{3}}\tau+z|}=e_z+O((-t_1)^{-\frac{1}{3}}), \text{ provided that $|\tau|$ remains bounded}.
\end{equation}
After changing variables (cf. also \eqref{eqderder}), \eqref{weakaa} becomes:
\begin{equation}\label{weakaab}
\int_{\tilde S^l}\frac{1}{2}\big(-\frac{3}{2}t_1\big)^{\frac{4-n}{3}}[B( \tilde y)+C(\tilde y)]\leq  \int_{\tilde S^l}\frac{1}{2}\big(-\frac{3}{2}t_1\big)^{\frac{4-n}{3}}[G(\tilde y,\tilde \chi^{-l})+R(\tilde y,\tilde \chi^{-l})], \ \forall l<l_0,
\end{equation}
with $B$ (resp. $C$) as in \eqref{bdef} (resp. \eqref{cdef}), and
$$G(\tilde y,\tilde \chi^{-l})=\Big[\frac{1}{2}\sum_{i=1}^n| \tilde y_{t_i}+ \tilde \chi^{-l}_{t_i}\xi|^2-\frac{|\tilde y +\tilde \chi^{-l}\xi |^2}{2}+\frac{|\tilde y +\tilde \chi^{-l}\xi |^4}{4}\Big],$$ 
\begin{align*}
R(\tilde y,\tilde \chi^{-l})&=\big(-\frac{3}{2}t_1\big)^{-\frac{1}{3}}\tilde \chi^{-l}\sum_{i=2}^n (\tilde y_{t_i}+\tilde\chi^{-l}_{t_i}\xi)\cdot\Big(\frac{e_i}{|(-\frac{3}{2}t_1)^{-\frac{1}{3}}\tau+z|}-\frac{t_i(-\frac{3}{2}t_1)^{-\frac{1}{3}}((-\frac{3}{2}t_1)^{-\frac{1}{3}}\tau+z)}{|(-\frac{3}{2}t_1)^{-\frac{1}{3}}\tau+z|^3}\Big)\\
&+\big(-\frac{3}{2}t_1\big)^{-\frac{2}{3}}\frac{|\tilde \chi^{-l}|^2}{2}\sum_{i=2}^n\Big|\frac{e_i}{|(-\frac{3}{2}t_1)^{-\frac{1}{3}}\tau+z|}-\frac{t_i(-\frac{3}{2}t_1)^{-\frac{1}{3}}((-\frac{3}{2}t_1)^{-\frac{1}{3}}\tau+z)}{|(-\frac{3}{2}t_1)^{-\frac{1}{3}}\tau+z|^3}\Big|^2\\
&-\frac{1}{2}\big(-\frac{3}{2}t_1\big)^{-1}(\tilde y_{t_1}+\tilde\chi^{-l}_{t_1}\xi)\cdot\big(\tilde y+\tilde\chi^{-l}\xi+\sum_{i=2}^nt_i(\tilde y_{t_i}+\tilde \chi^{-l}_{t_i}\xi)\big)\\
&+\frac{1}{8}\big(-\frac{3}{2}t_1\big)^{-2}\Big|\tilde y+\tilde\chi^{-l}\xi-\sum_{i=2}^nt_i(\tilde y_{t_i}+\tilde \chi^{-l}_{t_i}\xi)\Big|^2
\end{align*}
Since estimates \eqref{eqder8} hold in view of the bound $|y(x)|\leq h(x_1)$ provided by Lemma \ref{lem1} (iii), we have
$R(\tilde y,\tilde \chi^{-l})=O((-t_1)^{-\frac{1}{3}})$. 
On the other hand, \eqref{ksi} and \eqref{eqder8} imply that
\begin{equation}\label{eqq1}
G(\tilde y,\tilde \chi^{-l})=\Big[\frac{1}{2}\sum_{i=1}^n| \tilde y_{t_i}+ \tilde \chi^{-l}_{t_i}e_z|^2-\frac{|\tilde y +\tilde \chi^{-l}e_z |^2}{2}+\frac{|\tilde y +\tilde \chi^{-l}e_z |^4}{4}\Big]+O((-t_1)^{-\frac{1}{3}}),
\end{equation}
or equivalently
\begin{equation}\label{eqq2}
G(\tilde y,\tilde \chi^{-l})=B(\tilde y+\tilde \chi^{-l}e_z)+O((-t_1)^{-\frac{1}{3}}).
\end{equation}
Finally, given a sequence $l_k\to -\infty$, one can show as in the proof of Lemma \ref{ass}, that up to subsequence, $\tilde y_k(t_1,t_2,\ldots,t_n):=\tilde y(t_1+l_k,t_2,\ldots,t_n)$ converges in $C^1_{\mathrm{ loc}}(\R^n;\R^{n-1})$ to a solution $u(t_1,t_2,\ldots,t_n)$, $u:\R^n\to\R^{n-1}$, of $\Delta u=|u|^2-u$. Moreover, in view of Lemma \ref{lem1}, we have $u(t)=v(t)e_z$, with $v\geq 0$, since the vectors $\tilde y_k(t)$ and $f_k(t):=z+\big(-\frac{3}{2}(t_1+l_k)\big)^{-\frac{1}{3}}\tau\in\R^{n-1}$ have the same direction, and $\lim_{k\to\infty}f_k(t)=z$.
To conclude, we reproduce the argument at the end of the proof of Lemma \ref{ass}, and obtain that
\begin{align*}
\int_{\tilde S^{l_k}}\frac{1}{2}\big(-\frac{3}{2}t_1\big)^{\frac{4-n}{3}}[B( \tilde y)+C(\tilde y)]&=
(-c-l_k)^{\frac{4-n}{3}}\int_{\tilde S^{l_k}}B( \tilde y)+O((-c-l_k)^{\frac{1-n}{3}})\\
&=(-c-l_k)^{\frac{4-n}{3}}\int_{\tilde S}B( \tilde y_k)+O((-c-l_k)^{\frac{1-n}{3}}),
\end{align*}
while
\begin{align*}
\int_{\tilde{ S}^{l_k}}\frac{1}{2}\big(-\frac{3}{2}t_1\big)^{\frac{4-n}{3}}[G(\tilde{ y},\tilde{ \chi}^{-l_k})+
R(\tilde{ y},\tilde{ \chi}^{-l_k})]&=(-c-l_k)^{\frac{4-n}{3}}\int_{\tilde{ S}^{l_k}}B( \tilde{ y}+\tilde{ \psi}^{-l_k})+O((-c-l_k)^{\frac{3-n}{3}})\\
&=(-c-l_k)^{\frac{4-n}{3}}\int_{\tilde{ S}}B( \tilde{ y}_k+\tilde{ \chi} e_z)+O((-c-l_k)^{\frac{3-n}{3}}),
\end{align*}
As a consequence, it follows from \eqref{weakaab} that
\begin{equation}\label{weq3}
\int_{\tilde S}B( \tilde y_k)+O((-c-l_k)^{-1})\leq\int_{\tilde S}B( \tilde y_k+\tilde \chi e_z)+O((-c-l_k)^{-\frac{1}{3}}),
\end{equation}
holds for $k$ large enough. Next, passing to the limit, we get
\begin{equation}\label{weq3}
E_{\mathrm{GL}}(v e_z, \tilde S)\leq E_{\mathrm{GL}}(v e_z+\tilde \chi e_z, \tilde S),
\end{equation}
i.e. $v:\R^n\to\R$ is a nonnegative minimal solution of $\Delta v=v^3-v$. Therefore, in view of  \cite[Corollary 5.2]{book}, we deduce that $v \equiv 1$, and since the limit of $\tilde y_k$ is independent of the sequence $l_k\to-\infty$, we have established that $\lim_{l\to\-\infty}\tilde y(t_1+l,t_2,\ldots,t_n)=e_z$. This completes the proof of Lemma \ref{lem44}.
\end{proof}

Next, we examine the asymptotic convergence of $y(x_1,z)$, as $|z|\to\infty$.

\begin{lemma}\label{ass22}
Let $y(x_1,z)=y_{\mathrm{rad}}(x_1,|z|)\frac{z}{|z|}$ be the solution provided by Lemma \ref{lem1}, and let $\{z_k\}\subset\R^{n-1}$ be a sequence such that  $\lim_{k\to\infty}|z_k|=\infty$, and $\lim_{k\to\infty}\frac{z_k}{|z_k|}=\n_0$.
Then,  $y_k(x_1,z):=y(x_1,z+z_k)$ converges as $k\to\infty$, to $h(x_1)\n_0$ in $C^1_{\mathrm{loc}}(\R^n;\R^{n-1})$. In particular, $\lim_{l\to \infty}  y_{\mathrm{rad}}(x_1,\sigma+l)=h(x_1)$ for the $C^1_{\mathrm{loc}}(\R^2;\R)$ convergence.
\end{lemma}
\begin{proof}
In view of the bound $|y(x_1,z)|\leq h(x_1)$ provided by Lemma \ref{lem1} (iii), we obtain by the theorem of Ascoli that (up to subsequence) $ y_k$ converges in $C^1_{\mathrm{loc}}(  \R^n;\R^{n-1})$ to a solution $y_\infty:\R^n\to\R^{n-1}$ of \eqref{painhom}. In addition, since $\ y_k(x_1,z)=y_{\mathrm{rad}}(x_1,|z+z_k|)\frac{z+z_k}{|z+z_k|}$, and $\lim_{k\to\infty} \frac{z+z_k}{|z+z_k|}=\n_0$, we deduce that $ y_\infty(x_1,z)=\tilde y(x)\n_0$, with $\tilde y:\R^n\to\R$ a solution of \eqref{scalarpde}. 
By construction, $\tilde y$ is nonnegative. We are going to show that $\tilde y$ is also minimal. Indeed, given $\chi\in C^\infty_0(\R^n;\R)$, let $\psi_k(x_1,z):=\chi(x_1,z-z_k)\frac{z}{|z|}$ and $\phi_k(x_1,z):=\chi(x_1,z)\frac{z+z_k}{|z+z_k|}$. In view of Lemma \ref{lem1} (iv), we have 
$E_{\mathrm{P_{II}}}(y, \supp \psi_k)\leq E_{\mathrm{P_{II}}}(y+\psi_k, \supp \psi_k)$, or equivalently $E_{\mathrm{P_{II}}}(y_k, \supp \chi)\leq E_{\mathrm{P_{II}}}(y_k+\phi_k, \supp \chi)$. Next, passing to the limit, we obtain 
$E_{\mathrm{P_{II}}}(\tilde y \n_0, \supp \chi)\leq E_{\mathrm{P_{II}}}((\tilde y +\chi )\n_0, \supp \chi)$ i.e. $\tilde y$ is a minimal solution of \eqref{scalarpde}. Thus, since $\tilde y$ clearly satisfies $\tilde y(x)\leq h(x_1)$, we deduce from Lemma \ref{lema} that $\tilde y(x)=h(x_1)$, $\forall x\in\R^n$. Moreover, since this limit is uniquely determined, it is independent of the subsequence extracted from $\{y_k\}$. Finally, setting $y=(y_2,\ldots,y_n)\in\R^{n-1}$, we have 
$y_{\mathrm{rad}}(x_1,\sigma)=y_2(x_1,\sigma,0,\ldots,0)$, 
$\frac{\partial y_{\mathrm{rad}}}{\partial x_1}(x_1,\sigma)=\frac{\partial y_2}{\partial x_1}(x_1,\sigma,0,\ldots,0)$, and
$\frac{\partial y_{\mathrm{rad}}}{\partial \sigma}(x_1,\sigma)=\frac{\partial y_2}{\partial x_2}(x_1,\sigma,0,\ldots,0)$.
Therefore, $\lim_{l\to \infty}  y_{\mathrm{rad}}(x_1,\sigma+l)=h(x_1)$ holds in $C^1_{\mathrm{loc}}(\R^2;\R)$, according to what precedes.
\end{proof}

To establish the monotonicity properties stated in Theorem \ref{corpain2} (ii), we shall work with the projection $y_2$ of  the solution $y=(y_2,\ldots, y_{n})\in\R^{n-1}$ provided by Lemma \ref{lem1}. We shall first compute in Lemmas \ref{lll1} and \ref{lll2}, bounds for $\frac{\partial y_2}{\partial x_1}$ and $\frac{\partial y_2}{\partial x_2}$ when $x_1$ is large enough and $x_2>0$. Our main tool is a version of the maximum principle in unbounded domains \cite[Lemma 2.1]{beres}. We also utilize the asymptotic behaviour of $y$ provided by Lemmas \ref{lem44} and \ref{ass22}. Next, these bounds are extended to the whole space by applying the moving plane method (cf. Lemma \ref{movingplane}).

\begin{lemma}\label{lll1}
Let $y$ be the solution provided by Lemma \ref{lem1}. Then, we have $\frac{\partial y_{\mathrm{rad}}}{\partial x_1}(x_1,|z|)<0$, $\forall x_1\geq 0$, $\forall z\neq 0$. In addition, for every $d>0$, there holds $\sup_{|z|\geq d}  \frac{\partial y_{\mathrm{rad}}}{\partial x_1}(1,|z|)<0$, and  $\inf_{|z|\geq d}  y_{\mathrm{rad}}(1,|z|)>0$.
\end{lemma}

\begin{proof}
Given $\lambda\geq 0$, we define the function $\psi_\lambda(x_1,z):=y_2(x_1,z)-y_2(-x_1+2\lambda,z)$ for $x\in D_\lambda:=\{ (x_1,\ldots,x_n): x_1 > \lambda , x_2 >0\}$.  One can check that
$\psi_\lambda=0$ on $\partial D_\lambda$, and $$\Delta \psi_\lambda-c(x)\psi_\lambda= 2(x_1-\lambda)y_2(-x_1+2\lambda,z)\geq 0 \text{ on } D_\lambda,$$ with $c(x)=x_1+2(|y(x)|^2+|y(x)||y(-x_1+2\lambda,z)|+|y(-x_1+2\lambda,z)|^2)\geq 0$.
Furthermore, $\psi_\lambda$ is bounded above and not identically zero (cf. Lemma \ref{lem1} (iii) and Lemma \ref{lem44}).
As a consequence of the maximum principle (cf. \cite[Lemma 2.1]{beres}), it follows that
$\psi_{\lambda}(x)< 0$, $\forall x_1> \lambda$, $\forall x_2> 0$, $\forall (x_3,\ldots,x_n)\in\R^{n-2}$, and thus by Hopf's Lemma we have $\frac{\partial \psi_\lambda}{\partial x_1}(\lambda,z)=2\frac{\partial y_2}{\partial x_1}(\lambda,x_2,\ldots,x_n)<0$, provided that $ x_2>0$. This proves that $\frac{\partial y_{\mathrm{rad}}}{\partial x_1}(x_1,|z|)<0$, $\forall x_1\geq 0$, $\forall z\neq 0$. Finally, Lemma \ref{ass22} implies that $\lim_{\sigma\to\infty}y_{\mathrm{rad}}(1,\sigma)=h(1)$, and $\lim_{\sigma\to\infty}\frac{\partial y_{\mathrm{rad}}}{\partial x_1}(1,\sigma)=h'(1)$. Therefore, it is clear that $\sup_{|z|\geq d}  \frac{\partial y_{\mathrm{rad}}}{\partial x_1}(1,|z|)<0$, as well as $\inf_{|z|\geq d}  y_{\mathrm{rad}}(1,|z|)>0$ hold.
\end{proof}

\begin{lemma}\label{lll2}
Let $y=(y_2,\ldots,y_m)\in\R^{n-1}$ be the solution provided by Lemma \ref{lem1}. Then, for every vector $\n=(\cos(\theta+\frac{\pi}{2}), \sin(\theta+\frac{\pi}{2}),0,\ldots,0)\in\R^n$, with $\theta \in (0,\frac{\pi}{2})$, there exists $s_\n>0$ such that $\nabla y_2(x)\cdot \n>0$ holds, provided that $ x_1 >s_\n$, and  $x_2 >0$.
\end{lemma}

\begin{proof}
Let $(e_2,\ldots,e_n)$ be the canonical basis of $\R^{n-1}$.
Our first claim is that there is a constant $k_1>0$, such that $k_1 \frac{\partial y_2}{\partial x_1} (x)\leq -\sqrt{x_1} y_2(x)$, provided that $x_1\geq 1$, and $x_2\geq 0$. Indeed, let
\begin{equation}\label{psi9}
\psi(x)=k_1 \frac{\partial y_2}{\partial x_1} (x)+\sqrt{x_1}  y_2( x)=\Big(k_1\frac{\partial y_{\mathrm{rad}}}{\partial x_1}(x_1,|z|)+\sqrt{x_1} y_{\mathrm{rad}}(x_1,|z|)\Big)\frac{x_2}{|z|}, 
\end{equation}
for $x\in D:=\{x:\R^n:x_1>1, x_2>0\}$, where the constant $k_1>0$ will be adjusted later. It is clear that $\psi$ vanishes on the hyperplane $x_2=0$. We also notice that $\frac{\partial^2 y_2}{\partial x_1\partial x_2}(1,0,x_3,\ldots,x_n)<0$ by Hopf's Lemma, since the function $\frac{\partial y_2}{\partial x_1}$ vanishes on $\{x_1=1,x_2=0\}$, is negative on $\{ x_1> 0,  x_2> 0\}$, and satisfies 
\begin{equation}\label{eqqq11b}
\Delta \frac{\partial y_2}{\partial x_1}= y_2 +(x_1+6 | y|^2)\frac{\partial y_2}{\partial x_1}\geq(x_1+6|y|^2)\frac{\partial y_2}{\partial x_1} \text{ on } D.
\end{equation}
As a consequence, when $k_1$ is large enough, there exists $d>0$ such that $\psi(1,z)\leq 0$, provided that $|z|\leq d$ and $x_2\geq 0$. In addition, \eqref{psi9} and $\sup_{|z|\geq d}  \frac{\partial y_{\mathrm{rad}}}{\partial x_1}(1,|z|)<0$, $\forall d>0$, imply that when $k_1$ is large enough, we have $\psi(1,x_2,\ldots,x_n)\leq 0$, $\forall x_2\geq 0$, $\forall (x_3,\ldots,x_n)\in\R^{n-2}$. Next, we compute
\begin{align*}
\Delta \psi&=\Big(x_1+6|y|^2+\frac{1}{k_1\sqrt{x_1}}\Big)k_1\frac{\partial y_2}{\partial x_1}+\Big(x_1+2|y|^2+\frac{k_1}{\sqrt{x_1}}-\frac{1}{4x_1^2}\Big) \sqrt{x_1}y_2\\
&=\Big(x_1+2|y|^2+\frac{k_1}{\sqrt{x_1}}-\frac{1}{4x_1^2}\Big) \psi+\Big(4|y|^2+\frac{1}{k_1\sqrt{x_1}}-\frac{k_1}{\sqrt{x_1}}+\frac{1}{4x_1^2}\Big) k_1 \frac{\partial y_2}{\partial x_1}.
\end{align*}
By choosing $k_1$ large enough we can ensure that $\big(x_1+2|y|^2+\frac{k_1}{\sqrt{x_1}}-\frac{1}{4x_1^2}\big) \geq 0$ and $\big(4|y|^2+\frac{1}{k_1\sqrt{x_1}}-\frac{k_1}{\sqrt{x_1}}+\frac{1}{4x_1^2}\big)\leq 0$, when $x_1\geq 1$, and $x_2\geq 0$. Thus, our claim follows from the maximum principle (cf. \cite[Lemma 2.1]{beres}).

Similarly, we are going to establish that there is a constant $k_2>0$, such that $\frac{\partial y_2}{\partial x_2} (x)\geq -k_2 y_2(x)$, provided that $x_1\geq 1$, and $x_2\geq 0$. 
To do this we let
\begin{equation}\label{psipsi}
\psi(x)=- \frac{\partial y_2}{\partial x_2}(x)-k_2 y_2(x) =-\frac{x_2}{|z|}\Big(\frac{\partial y_{\mathrm{rad}}}{\partial \sigma}(x_1,|z|)\frac{x_2}{|z|}+k_2y_{\mathrm{rad}}(x_1,|z|)\Big) -y_{\mathrm{rad}}(x_1,|z|)\Big(\frac{|z|^2-x_2^2}{|z|^3}\Big)\text{ for $x\in D$},
\end{equation}
where the constant $k_2$ will again be adjusted later. We first notice that
$\frac{\partial y_2}{\partial x_2}(x_1,0,x_3,\ldots,x_n)>0$ holds on the hyperplane $x_2=0$, since the function $y_2$ vanishes on the hyperplane $x_2=0$, is positive in $\{x_2> 0\}$, and satisfies $\Delta y_2=(x_1+2|y|^2)y_2$. 
As a consequence, when $k_2$ is large enough, there exists $d>0$ such that $\psi(1,z)\leq 0$, provided that $|z|\leq d$ and $x_2\geq 0$. In addition, \eqref{psipsi} and $\inf_{|z|\geq d}  y_{\mathrm{rad}}(1,|z|)>0$, $\forall d>0$, imply that when $k_1$ is large enough, we have $\psi(1,x_2,\ldots,x_n)\leq 0$, provided that $x_2\geq 0$. On the other hand, it is clear that $\psi(x_1,0,x_3,\ldots,x_n)<0$, $\forall x_1\geq 1$, $\forall (x_3,\ldots,x_n)\in\R^{n-2}$. Next, we compute successively for $x\in D$:
$$\frac{\partial y}{\partial x_2}(x)=\frac{\partial y_{\mathrm{rad}}}{\partial \sigma}(x_1,|z|)\frac{x_2}{|z|}\frac{z}{|z|}+y_{\mathrm{rad}}(x_1,|z|) \Big(\frac{e_2}{|z|}-\frac{x_2z}{|z|^3}\Big),$$
$$y(x)\cdot \frac{\partial y}{\partial x_2}(x)=\frac{\partial y_{\mathrm{rad}}}{\partial \sigma}(x_1,|z|) y_{\mathrm{rad}}(x_1,|z|)       \frac{x_2}{|z|} ,$$
$$\Big(y(x)\cdot \frac{\partial y}{\partial x_2}(x)\Big)y_2(x)=\frac{\partial y_{\mathrm{rad}}}{\partial \sigma}(x_1,|z|) y_{\mathrm{rad}}^2(x_1,|z|)       \frac{x_2^2}{|z|^2} \leq |y(x)|^2\frac{\partial y_2}{\partial x_2}(x),$$
$$\Delta \frac{\partial y_2}{\partial x_2}(x)=(x_1+2|y(x)|^2) \frac{\partial y_2}{\partial x_2}(x) +4 \Big( y(x)\cdot \frac{\partial y}{\partial x_2}(x)\Big)y_2(x)\leq (x_1+6|y(x)|^2) \frac{\partial y_2}{\partial x_2}(x),$$
$$\Delta y_2(x)=(x_1+2|y(x)|^2)y_2(x)\leq (x_1+6|y(x)|^2)y_2(x),$$
\begin{align*}
\Delta \psi&\geq (x_1+6|y|^2)\psi.
\end{align*}
Thus, it follows from the maximum principle that $\psi\leq 0$ in $D$. 
Finally, given $\theta\in (0,\pi/2)$, we have
\[\nabla y_2(x)\cdot \n=-\frac{\partial y_2}{\partial x_1}(x)\sin\theta+\frac{\partial y_2}{\partial x_2}(x) \cos\theta\geq\Big(\frac{\sqrt{x_1}}{k_1}\sin\theta-k_2  \cos\theta\Big) y_2(x), \ \forall x \in [1,\infty)\times[ 0,\infty)\times\R^{n-2},\]
and therefore $\nabla y_2(x)\cdot \n>0$ provided that $x_1>s_\n:=\big(\frac{k_1k_2}{\tan \theta}\big)^2$, and $x_2>0$.
\end{proof}

\begin{lemma}\label{movingplane}
Let $y=(y_2,\ldots,y_m)\in\R^{n-1}$ be the solution provided by Lemma \ref{lem1}, and let $\theta\in (0,\frac{\pi}{2})$ be fixed. For every $\lambda \in \R$, we consider the reflection $\rho_\lambda$ with respect to the hyperplane $\Gamma_\lambda:=\{x\in\R^n: x_2=\tan\theta (x_1-\lambda)\}$, and the domain
$D_\lambda:=\{x\in\R^n: 0<x_2<\tan\theta (x_1-\lambda)\}$.
Then, the function $\psi_\lambda(x):=y_2(x)-y_2(\rho_\lambda(x))$
is negative in $D_\lambda$, for every $\lambda\in\R$.
\end{lemma}
\begin{proof}
We set $\n=(\cos(\theta+\frac{\pi}{2}), \sin(\theta+\frac{\pi}{2}),0,\ldots,0)\in\R^n$, 
and denote by $(p',q',\zeta)\in\R\times\R\times\R^{n-2}$ the image by $\rho_\lambda$ of a point $(p,q,\zeta)\in D_\lambda$, and by $D'_\lambda$ the set $\rho_\lambda(D_\lambda)$.
It is obvious that $\psi_\lambda(x_1,0,\zeta)< 0$, $\forall x_1> \lambda$, $\forall \zeta\in\R^{n-2}$, and that $\psi_\lambda(x)= 0$, $\forall x\in \Gamma_\lambda$. 
We can also check that for $(p,q,\zeta)\in D_\lambda$, we have $p>p'$ and $q<q'$, as well as:
\begin{subequations}
\begin{equation}\label{eqa1}
y_2(p,q,\zeta)=|y(p,q,\zeta)|\frac{q}{\sqrt{q^2+|\zeta|^2}},
\end{equation}
\begin{equation}\label{eqa2}
y_2(p',q',\zeta)=|y(p',q',\zeta)|\frac{q'}{\sqrt{(q')^2+|\zeta|^2}},
\end{equation}
\begin{equation}\label{eqa3}
\frac{q}{\sqrt{q^2+|\zeta|^2}}\leq\frac{q'}{\sqrt{(q')^2+|\zeta|^2}},
\end{equation}
\begin{equation}\label{eqa4}
\Delta \psi_\lambda(p,q,\zeta)-c(p,q,\zeta)\psi_\lambda(p,q,\zeta)\geq 0,
\end{equation}
\end{subequations}
with 
$$c(p,q,\zeta)=
p+2     (|y(p,q,\zeta)|^2+|y(p,q,\zeta)||y(p',q',\zeta)|+|y(p',q',\zeta)|^2).$$
Next, for each $\lambda\in\R$ we consider the statement
\begin{equation}\label{statem}
\psi_\lambda(p,q,\zeta) <0, \quad \forall (p,q,\zeta)\in D_\lambda.
\end{equation}

We shall first establish Lemma \ref{movingplane} in the case where $\theta\in (0,\frac{\pi}{4})$.
According to Lemma \ref{lll2}, \eqref{statem} is valid for each $\lambda> s_\n$.
Set $\lambda_0=\inf\{\lambda\in \R: \psi_\mu<0 \text{ holds in $ D_\mu$, for each $\mu\geq \lambda$} \}$. We will prove $\lambda_0=-\infty$.
Assume instead $\lambda_0\in\R$. Then, there exist a sequence $\lambda_k<\lambda_0$ such that $\lim_{k\to\infty}\lambda_k=\lambda_0$, and a sequence
$(p_k,q_k,\zeta_k)\in D_{\lambda_k}$, such that 
\begin{equation}\label{assume}
y_2(p_k,q_k,\zeta_k)\geq y_2(p'_k,q'_k,\zeta_k), \forall k.
\end{equation}
According to Lemma \ref{lll2}, we have $p'_k\leq s_\n$, thus the sequence $(p_k,q_k)\subset\R^2$ is bounded, since by assumption $\theta\in(0,\pi/4)$.
Up to subsequence we may assume that $\lim_{k\to\infty}(p_k,q_k)=(p_0,q_0)$, with $p'_0\leq s_\n$.

We first examine the case where up to subsequence $\lim_{k\to\infty}\zeta_k=\zeta_0\in\R^{n-2}$. Note that $(p_0,q_0,\zeta_0)\in\overline{D_{\lambda_0}}$.
By definition of $\lambda_0$, we have $\psi_{\lambda_0}\leq 0$ in $D_{\lambda_0}$, and $\psi_{\lambda_0}(p_0,q_0,\zeta_0)=0$ i.e. $y_2(p_0,q_0,\zeta_0)=y_2(p'_0,q'_0,\zeta_0)$.
Now we distinguish the following cases. If $(p_0,q_0,\zeta_0)\in D_{\lambda_0}$, the maximum principle implies that $\psi_{\lambda_0}\equiv 0$ in $D_{\lambda_0}$. Clearly, this situation is excluded, since $y_2$ is positive
in the half-space $\{x_2>0\}$.
On the other hand, the maximum principle also implies that $\frac{\partial \psi_{\lambda_0}}{\partial \n}(p,q,\zeta)=2 \frac{\partial y_2}{\partial \n}(p,q,\zeta)>0$, provided that $(p,q,\zeta)\in \Gamma_{\lambda_0}$ and $q>0$.
Furthermore, the previous inequality still holds on the subspace $\{p=\lambda_0\}\cap\{q=0\}$, since $\frac{\partial y_{2}}{\partial x_2}(x_1,0,\zeta)>0$ and $\frac{\partial y_2}{\partial x_1}(x_1,0,\zeta)=0$ hold, for every $ x_1\in\R$, and $ \zeta\in\R^{n-2}$ (cf. the proof of Lemma \ref{lll2}).
As a consequence,$(p_0,q_0,\zeta_0)$ 
cannot belong to $\Gamma_{\lambda_0}$. Finally, since the case where $p_0>\lambda_0$ and $q_0=0$ is ruled out (because $y_2$ is positive in the half-plane $\{x_2>0\}$), we have reached a contradiction.

On the other hand, when $\lim_{k\to\infty}|\zeta_k|=\infty$, we have in view of \eqref{eqa1} and \eqref{eqa2}:
\begin{equation}
\frac{  y_2(p_k,q_k,\zeta_k)               }{y_2(p'_k,q'_k,\zeta_k)}=\frac{|y(p_k,q_k,\zeta_k)|q_k\sqrt{(q'_k)^2+|\zeta_k|^2}}{|y(p'_k,q'_k,\zeta_k)|q'_k\sqrt{q^2_k+|\zeta_k|^2}}\sim\frac{|y(p_k,q_k,\zeta_k)|q_k}{|y(p'_k,q'_k,\zeta_k)|q'_k},\text{ as $k\to\infty$}. 
\end{equation}
In addition, it follows from Lemma \ref{ass22}, that $\lim_{k\to\infty}\frac{|y(p_k,q_k,\zeta_k)|}{|y(p'_k,q'_k,\zeta_k)|}=\frac{h(p_0)}{h(p'_0)}$. Thus, since $p_0\geq p'_0$, $h(p'_0)\geq h(p_0)$, and $q'_0\geq q_0$ hold, the assumption $y_2(p_k,q_k,\zeta_k)\geq y_2(p'_k,q'_k,\zeta_k)$ implies that $p_0=p'_0$ and $q_0=q'_0$. Then, we notice that, $$ y_2(p'_k,q'_k,\zeta_k)- y_2(p_k,q_k,\zeta_k)=\sqrt{(p_k-p'_k)^2+(q_k-q'_k)^2}\frac{\partial y_2}{\partial \n}((p_k,q_k,\zeta_k)+t_k \n),$$ with $t_k\in(0,\sqrt{(p_k-p'_k)^2+(q_k-q'_k)^2})$, and 
$\frac{\partial y_2}{\partial \n}(x)\geq \frac{x_2}{|z|}\big(-\sin\theta\frac{\partial y_{\mathrm{rad}}}{\partial x_1}(x_1,|z|)+\cos\theta\frac{x_2}{|z|}\frac{\partial y_{\mathrm{rad}}}{\partial \sigma}(x_1,|z|)\big)$.
Setting $(p_k,q_k,\zeta_k)+t_k \n=:(\tilde p_k,\tilde q_k,\tilde \zeta_k)$, we deduce again from Lemma \ref{ass22} that 
\begin{align*}
\frac{\partial y_2}{\partial \n}(\tilde p_k,\tilde q_k,\tilde \zeta_k)\geq \frac{\tilde q_k}{\sqrt{\tilde q_k^2+|\tilde\zeta_k|^2}}(-h'(p_0)\sin\theta +o(1)).
\end{align*}
Therefore, we obtain $ y_2(p'_k,q'_k,\zeta_k)- y_2(p_k,q_k,\zeta_k)>0$, for $k$ large enough, which is a contradiction.
\begin{figure}[h]
\includegraphics{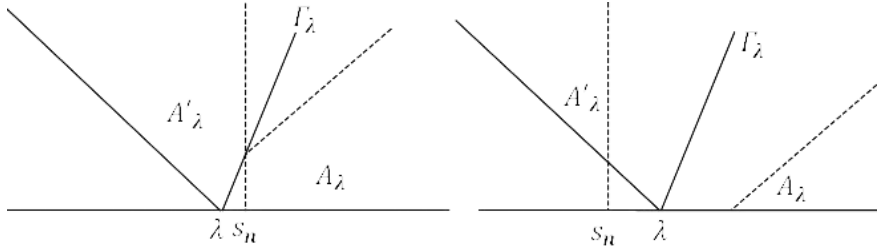}
\caption{The sets $A_\lambda$, $A'_\lambda$, and the subspace $\Gamma_\lambda$, in the cases where $\lambda<s_\n$ and $\lambda>s_\n$. }
\label{fig}
\end{figure}

Next, we establish Lemma \ref{movingplane} in the case where $\theta \in [\frac{\pi}{4},\frac{\pi}{2})$.  When $\theta=\frac{\pi}{4}$, it is clear that
\eqref{statem} is valid for each $\lambda>s_\n$. Otherwise, when $\theta \in (\frac{\pi}{4},\frac{\pi}{2})$, let $A'_\lambda:=\{(p',q',\zeta)\in D'_\lambda: p'\leq s_\n\}$, and let $A_\lambda=\rho_\lambda(A'_\lambda)$. Our first claim is that $m:=\inf_{A'_{s_\n+1}}|y| >0$. Indeed, proceeding as in the proof of Lemma \ref{lem44}, one can see that
\[
\lim_{x \in A'_{s_\n+1},x_1\to-\infty}\frac{\sqrt{2}}{\sqrt{-x_1}}|y(x)|=1.
\]
In addition, according to Lemma \ref{ass22}, we obtain that $\inf\{|y(x)|: x\in A'_{s_\n+1}, \, s_\n-l\leq x_1\leq s_\n\}>0$, for every constant $l>0$. Thus, $m>0$. On the other hand, we have $\lim_{\lambda\to\infty}\sup\{| y(x)|: x\in A_\lambda\}=0$, since $\lim_{\lambda\to\infty}\inf\{ x_1: x\in A_\lambda\}=0$. As a consequence when $\lambda\geq s_\n+1$ is large enough, we have 
\begin{equation}
\frac{  y_2(p,q,\zeta)               }{y_2(p',q',\zeta)}\leq \frac{|y(p,q,\zeta)|}{|y(p',q',\zeta)|}<1, \ \forall (p,q,\zeta)\in A_\lambda,
\end{equation}
and also $y_2(p',q',\zeta)<y_2(p,q,\zeta)$, $\forall (p,q,\zeta)\in D_\lambda\setminus A_\lambda$, by definition of $s_\n$. This establishes that \eqref{statem} holds for $\lambda$ large enough.
Then, defining $\lambda_0$ as previously, we assume by contradiction that $\lambda_0\in\R$, and there exist sequences $\lambda_k\to\lambda_0$ and $(p_k,q_k,\zeta_k)\in D_{\lambda_k}$ satisfying \eqref{assume}. We need to show that $(p_k,q_k)$ is bounded in $\R^2$. 
Indeed, if $\lim_{k\to\infty}p_k=\infty$, then we also have $\lim_{k\to\infty}q'_k=\infty$, as well as $p'_k\leq s_\n$, in view of \eqref{assume} and the definition of $s_n$. In particular, it follows from Lemma \ref{lem44} and Lemma \ref{ass22} (resp. from the bound $|y(x)|\leq h(x_1)$) that
$\liminf_{k\to\infty}|y(p'_k,q'_k,\zeta_k)|\geq h(s_\n)$ (resp. $\lim_{k\to\infty}|y(p_k,q_k,\zeta_k)|=0$). 
As a consequence, we obtain
\begin{equation}
\frac{  y_2(p_k,q_k,\zeta_k)               }{y_2(p'_k,q'_k,\zeta_k)}\leq \frac{|y(p_k,q_k,\zeta_k)|}{|y(p'_k,q'_k,\zeta_k)|}\to 0, \text{ as $k\to\infty$},
\end{equation}
which contradicts \eqref{assume}. Now that the boundedness of the sequence $(p_k,q_k)$ is established, to complete the proof we reproduce the arguments detailed in the case where $\theta\in (0,\frac{\pi}{4})$.

\end{proof}

Lemma \ref{movingplane} implies that $\forall \theta \in (0,\frac{\pi}{2})$, $\forall \lambda \in \R$, and $(p,q,\zeta)\in \Gamma_{\lambda}$ with $q>0$, we have
$\frac{\partial \psi_{\lambda}}{\partial \n}(p,q,\zeta)=2\frac{\partial y_2}{\partial \n}(p,q,\zeta)>0$, where $\n=(\cos(\theta+\frac{\pi}{2}), \sin(\theta+\frac{\pi}{2}),0,\ldots,0)$.
It follows that $\frac{\partial y_2}{\partial x_1}(x)\leq 0$, and $\frac{\partial y_2}{\partial x_2}(x)\geq 0$, provided that $x_2\geq 0$.
Moreover, in the half-space $x_2\geq 0$, $\frac{\partial y_2}{\partial x_1}$ and $\frac{\partial y_2}{\partial x_2}$ satisfy respectively $\Delta \frac{\partial y_2}{\partial x_1}\geq (x_1+6|y|^2)\frac{\partial y_2}{\partial x_1}$, and $\Delta \frac{\partial y_2}{\partial x_2}\leq (x_1+6|y|^2)\frac{\partial y_2}{\partial x_2}$, thus $\frac{\partial y_2}{\partial x_1}$ (resp. $\frac{\partial y_2}{\partial x_2}$) cannot vanish in the open half-space $x_2>0$, since otherwise we would obtain by the maximum principle $\frac{\partial y_2}{\partial x_1}\equiv 0$ (resp. $\frac{\partial y_2}{\partial x_2}\equiv 0$). These situations are excluded by Lemma \ref{lll1} and the fact that $y_2>0$ in the half-space $x_2>0$. Therefore we have proved that
$\frac{\partial y_{\mathrm{rad}}}{\partial x_1}(x_1,\sigma)=\frac{\partial y_2}{\partial x_1}(x_1,\sigma,0,\ldots,0)<0$, and $\frac{\partial y_{\mathrm{rad}}}{\partial \sigma}(x_1,\sigma)=\frac{\partial y_2}{\partial x_2}(x_1,\sigma,0,\ldots,0)>0$, provided that $\sigma>0$. 

Finally, we consider the rescaled map $\tilde y(t_1,\ldots,t_n)$ as in \eqref{ytilde}, with $z=0$, and proceed as in the proof of Lemma \ref{lem44}. Given a sequence $l_k\to -\infty$, one can see, that up to subsequence, $\tilde y_k(t_1,t_2,\ldots,t_n):=\tilde y(t_1+l_k,t_2,\ldots,t_n)$ converges in $C^1_{\mathrm{ loc}}(\R^n;\R^{n-1})$ to a solution $u(t_1,t_2,\ldots,t_n)$, $u:\R^n\to\R^{n-1}$, of $\Delta u=|u|^2-u$. Moreover, $u$ is by construction $O(n-1)$-equivariant with respect to $\tau:=(t_2,\ldots,t_n)$ (cf. \eqref{equivv}), and minimal for $O(n-1)$-equivariant perturbations. We also notice that \eqref{eqder3} and $y_{x_1}(x_1,r)\cdot \frac{r}{|r|}<0$, $\forall x_1\in\R$, $\forall r \in \R^{n-1}\setminus\{0\}$, imply that
$|t_1+l_k|^{\frac{2}{3}}\tilde y_{t_1}(t_1+l_k,t_2,\ldots,t_n)\cdot\frac{\tau}{|\tau|}+O(|t_1+l_k|^{-\frac{1}{3}})\leq 0$, $\forall t_1\in\R$, $\forall \tau\in\R^{n-1}\setminus\{0\}$. Passing to the limit as $k\to\infty$, we deduce that $u_{t_1}(t_1,\ldots,t_n)\cdot\frac{\tau}{|\tau|}\leq 0$, $\forall t_1\in\R$, $\forall \tau\in\R^{n-1}\setminus\{0\}$. As a consequence the limits $\lim_{t_1\to\pm\infty}u(t_1,t_2,\ldots,t_n)=:v^\pm(t_2,\ldots,t_n)$ exist, and one can see that $v^\pm:\R^{n-1}\to\R^{n-1}$ is an $O(n-1)$-equivariant solution of $\Delta v^\pm=|v^\pm|^2v^\pm-v^\pm$, which is minimal for $O(n-1)$-equivariant perturbations. That is, $v^\pm\equiv \eta$, where $\eta:\R^{n-1}\to\R^{n-1}$ is the standard vortex solution of the Ginzburg-Landau system \eqref{gl}. In addition, in view of the monotonicity of $u$ along the $t_1$ direction, we obtain that $u(t_1,\ldots,t_n)=\eta(t_2,\ldots,t_n)$, and since this limit is independent of the sequence $\{l_k\}$, we have established that \eqref{scale2} holds in the case where $z=0$. This completes the proof of Theorem \ref{corpain2}.

\section*{Acknowledgments}

The author would like to thank Micha{\l } Kowalczyk for several fruitful discusssions during a visit at the University of Chile. He was partially supported by the National Science Centre, Poland (Grant No. 2017/26/E/ST1/00817)

\end{document}